\documentclass[12pt,leqno]{amsart}

\usepackage{amscd}
\usepackage{amsthm}
\usepackage{color}
\usepackage{amssymb}
\usepackage{amsfonts}
\usepackage{latexsym}
\usepackage{verbatim}
\usepackage[utf8]{inputenc}

\theoremstyle{plain}
\newtheorem{proposition}{Proposition}

\newtheorem{theorem}[proposition]{Theorem}
\newtheorem{lemma}[proposition]{Lemma}
\newtheorem{corollary}[proposition]{Corollary}
\theoremstyle{definition}

\theoremstyle{remark}
\newtheorem{remark}{Remark}

\renewcommand{\b}{\begin{equation}}
\newcommand{\e}{\end{equation}}

\newcommand\R{{\mathbb R}}

\newcommand{\Gtwo}{\mathrm{G}_2}

\newcommand{\la}{\langle}
\newcommand{\ra}{\rangle}
\newcommand{\SU}{\mathrm{SU}}

 \title[Laplacian coflow on the Heisenberg group]{Laplacian coflow on the $7$-dimensional\\
 Heisenberg group}

\begin{document}

\author{Leonardo Bagaglini, Marisa Fern\'andez, Anna Fino}
\date{\today}
\address{Dipartimento di Matematica e Informatica \lq \lq Ulisse Dini\rq \rq \\ Universit\`a di Firenze\\
Viale Giovan Battista Morgagni, 67/A \\
50134 Firenze\\ Italy}
\email{leonardo.bagaglini@unifi.it}
 \address{Universidad del Pa\'{\i}s Vasco, Facultad de Ciencia y Tecnolog\'{\i}a, 
 Departamento de Matem\'aticas, Apartado 644, 48080 Bilbao, Spain}
  \email{marisa.fernandez@ehu.es}
\address{Dipartimento di Matematica \lq\lq Giuseppe Peano\rq\rq \\ Universit\`a di Torino\\
Via Carlo Alberto 10\\
10123 Torino\\ Italy}
 \email{annamaria.fino@unito.it}

\begin{abstract} 
We study  the  Laplacian coflow and the modified Laplacian coflow of $\Gtwo$-structures  
 on  the $7$-dimensional Heisenberg group. For the Laplacian coflow we show that  the solution is always ancient, 
that is it is  defined in some interval $(-\infty,T)$, with $0<T<+\infty$.
However, for the modified Laplacian coflow, we prove that in some cases the solution is defined only on a finite interval 
while in other cases the solution is 
ancient or eternal, that is it is defined on $(-\infty, \infty)$.
\end{abstract}

\maketitle

\section{Introduction}

A $7$-dimensional manifold $M$ carries a $\Gtwo$-structure if $M$ admits 
a globally defined 3-form $\varphi$, which is called $\Gtwo$ form,  that  can be described locally as
$$
\varphi = e^{127} + e^{347}+ e^{567} + e^{135} - e^{146} - e^{236} - e^{245},
$$
with respect to some local basis  $\{e^1, \ldots , e^7\}$  of the 1-forms  on $M$. Here, $e^{127}$ stands for 
$e^1\wedge e^2\wedge e^7$, and so on.
Such a 3-form $\varphi$ determines a Riemannian metric $g_{\varphi}$ and an orientation on $M$. 
 If $\nabla$ denotes  the Levi-Civita connection of $g_{\varphi}$, one  can view 
$\nabla \varphi$ as the torsion of the $\Gtwo$-structure $\varphi$. Thus, if  $\nabla \varphi=0$, which is equivalent
to $d \varphi =0$ and $d \star_{ \varphi} \varphi= 0$,  where $\star_{ \varphi}$ is the Hodge star operator with respect to 
$g_{\varphi}$, one says that  the $\Gtwo$-structure is torsion-free.

The  different  classses of $\Gtwo$-structures can be described in terms  of the exterior
derivatives $d\varphi$ and $d \star_{ \varphi} \varphi$ \cite{Bryant,FG}.  
If $d \varphi =0$,  then the $\Gtwo$-structure is called {\it closed} (or {\it calibrated} in the sense of Harvey and Lawson \cite{HL}) and 
if $\varphi$ is coclosed, that is if $\star_{ \varphi} \varphi$ is closed, then the $\Gtwo$-structure is called {\it coclosed} (or {\it cocalibrated} \cite{HL}).

 Since Hamilton introduced the Ricci flow in 1982  \cite{Hamilton}, 
geometric flows have been an important tool in studying geometric structures on manifolds.
The Laplacian flow for closed $\Gtwo$-structures on a 7-manifold $M$ has been introduced 
 by Bryant in \cite{Bryant}, and it  is given by 
 $$
  \left \{ \begin{array}{l} \frac{\partial} {\partial t} \varphi(t) = \Delta_{t}\,\varphi(t),\\[3pt]
d\,\varphi(t) =0,\\[3pt] 
\varphi(0) = \varphi,
\end{array} \right.
$$
where $\varphi(t)$ is a closed $\Gtwo$ form on $M$, $\Delta_{t} = d\,d^* + d^* d$  is the
Hodge Laplacian operator associated with the metric $g_{\varphi (t)}$ induced by the $3$-form $\varphi(t)$, and  $\varphi$
is the initial closed $\Gtwo$-structure.
A  short-time existence and uniqueness  for this flow, in the case of compact manifolds,
has been proved in  \cite{BF}. Regarding the long-time behavior of the Laplacian flow on compact manifolds $M$, 
 Lotay and Wei in \cite{LW2} have proved recently 
 that if the initial closed $\Gtwo$ form $\varphi$ is such that  its torsion is  sufficiently small (in a suitable sense), then
the Laplacian flow of $\varphi$ will exist for all time and converge to a torsion-free $\Gtwo$-structure.
Non-compact examples 
where the flow converges to a flat $\Gtwo$-structure have been given in \cite{FFM}.

Shi-type derivative estimates for the Riemann curvature tensor
 and torsion tensor  along the Laplacian  flow have been determined  in \cite{LW}, and 
 in \cite{Lotay} it is proved that for each fixed positive time $t \in (0,T]$, $(M,\varphi(t),g_{\varphi(t)})$ is real analytic. 
 Consequently, any Laplacian soliton is real analytic.
 Moreover, solitons of  the Laplacian flow of $\Gtwo$-structures  in the homogeneous case  have been studied  
recently by Lauret in \cite{Lauret3}  using the bracket flow and the algebraic soliton approach.

Some work has also been done on other related flows of $\Gtwo$-structures - such as the {\em Laplacian coflow}, or {\em flow}, for
coclosed $\Gtwo$-structures. This coflow has been originally proposed by Karigiannis,
McKay and Tsui in \cite{KMT} and, for an initial coclosed $\Gtwo$ form $\varphi$ with $\psi\,=\,\star_{\varphi} \varphi$, it is given by
\begin{equation} \label{co-flow} 
\frac{\partial} {\partial t} \psi(t) = -  \Delta_{t} \psi(t), \quad    d\psi(t)\,=\,0, \quad  \psi(0) = \psi,
\end{equation}
where $\psi(t)$ is the  Hodge dual 4-form of a $\Gtwo$-structure $\varphi(t)$, 
that is $\psi(t)\,=\, \star_t  \varphi(t)$,
$\Delta_{t}$ is the Hodge Laplacian operator with respect to the Remannian metric $g_{\varphi(t)}$.
This flow preserves the condition of the $\Gtwo$-structure being coclosed, that is $\psi(t)$ is closed for any $t$,  and it  
was studied in \cite{KMT} for two explicit examples of coclosed $\Gtwo$-structures with symmetry, namely for  
warped products of an interval, or a circle, with a compact 6-manifold $N$ which is taken to be either a nearly K\"ahler manifold or a Calabi-Yau manifold. 
Nevertheless, in \cite{Grigorian} it was shown that the coflow \eqref{co-flow} is not even a weakly parabolic flow, and 
that the symbol of the operator $\Delta_{t}$, acting on $4-$forms, has a
mixed signature. But no general result is known about the short time existence of the coflow \eqref{co-flow}.

A {\em modified Laplacian coflow} was introduced by Grigorian in \cite{Grigorian}
\begin{equation} \label{modifiedcoflow}
\frac{\partial }{\partial t} \psi(t)  = \Delta_{t}  \psi(t) + 2d \Big(\big(A- {\rm{Tr}_t}(\tau(t))\big)\varphi(t)\Big), \quad    d\psi(t)\,=\,0, \quad \psi(0) = \psi,
\end{equation}
where ${\rm{Tr}_t}(\tau(t))$ is the trace of the full torsion tensor $\tau(t)$ of the $\Gtwo$-structure defined by $\varphi(t)$, and $A$ is a
fixed positive constant (see Section \ref{section-coflows} for the details). 
Moreover, in \cite{Grigorian} it is proved that the coflow \eqref{modifiedcoflow} is weakly parabolic in the
direction of closed forms  $\psi(t)$ up to diffeomorphisms  and, on compact manifolds, it has a unique solution $\psi(t)$
for the short time period $t\in [0, \epsilon)$, for some $\epsilon\,>0$.

 In \cite{BFF}, it is given a classification of 2-step nilpotent Lie groups admitting left invariant coclosed $\Gtwo$-structures.
In this paper, we study the coflows  \eqref{co-flow} and \eqref{modifiedcoflow}
in the case of the $7$-dimensional Heisenberg group $H$. 

As we mentioned before,  there is not known  any general result on the short time existence of solution for the coflow 
\eqref{co-flow}. Nevertheless, in Theorem \ref{Hflow}, we show that the solution of the coflow \eqref{co-flow} for  any 
 coclosed $\Gtwo$-structure  on the Heisenberg group is always {\em ancient}, 
that is it is defined on a time interval of the form $(-\infty, T)$, where $T > 0$ is a real number.  To our knowledge, these are the first examples of non-compact manifolds having a coclosed $\Gtwo$-structure for which the time interval of existence of the solution for  \eqref{co-flow} is not finite.
 However, we prove that the solution of the coflow \eqref{modifiedcoflow} for some coclosed $\Gtwo$ forms on $H$ is defined only on a finite interval 
(Theorem \ref{Th6.2}) and,
for other coclosed $\Gtwo$ forms, 
the solution of \eqref{modifiedcoflow} is {\em ancient} (Theorem \ref{Th6.1}, part i), and Theorem \ref{Th6.1-2}) or {\em eternal}, that is it is defined 
for all  $t \in \R$ (Theorem  \ref{Th6.1}, part ii)).

Moreover, considering the coflows  \eqref{co-flow}   and \eqref{modifiedcoflow} on the associated Lie algebra as a bracket flow on $\R^7$, in a similar way as Lauret did in \cite{La09} for the Ricci flow, we show that    the underlying metrics $g(t)$ of the  solution in Corollary  \ref{th5.1}  and  Theorem \ref{Th6.1-2} converge smoothly, up to pull-back by time-dependent diffeomorphisms, to a flat metric, as $t$ goes to infinity.  Indeed, by \cite[Proposition 2.8]{La09}  the convergence of the metrics in ${\mathcal C}^{\infty}$  uniformly on 
compact sets  in $\R^7$  is equivalent to the convergence of the nilpotent Lie brackets $\mu(t)$  in the algebraic subset of nilpotent Lie brackets 
${\mathcal N} \subset (\Lambda^2 \R^7)^* \otimes \R^7$  with the usual vector space topology.

\section{Coclosed $\Gtwo$-structures on the Heisenberg group}\label{coclosed-Heisenberg}

A $7$-dimensional  manifold $M$ is said to admit a $\Gtwo$-{\em structure} if there is a reduction of the structure group of 
its frame bundle from $\mathrm{GL}(7,\mathbb{R})$ to the  exceptional Lie group $\Gtwo$, which can actually be viewed naturally as a subgroup of $\mathrm{SO}(7)$. 
Thus, a $\Gtwo$-structure determines a Riemannian metric and an orientation on $M$. In fact, one can prove 
that the presence of a $\Gtwo$-structure is equivalent to the existence of a differential $3$-form $\varphi$ (the $\Gtwo$ form) on $M$, which 
 induces the Riemannian metric $g_{\varphi}$
given by
\begin{equation}\label{metric}
g_{\varphi} (X,Y)\, vol=\frac{1}{6} \iota_{X}\varphi \wedge \iota_{Y}\varphi \wedge \varphi,
\end{equation}
for any vector fields $X, Y$ on $M$, where $vol$ is the volume form on $M$, and $\iota_{X}$ denotes the contraction by $X$.
Let  $\star_{\varphi}$ be the Hodge star operator determined by $g_{\varphi}$ and the orientation induced by $\varphi$. We will always write 
$\psi$ to denote the dual 4-form of a $\Gtwo$-structure $\varphi$, that is
$$
\psi=\star_{\varphi} \varphi.
$$ 
A manifold $M$ has a {\em coclosed} (or {\em cocalibrated}) $\Gtwo$-{\em structure}
if there is a $\Gtwo$-structure on $M$ such that the $\Gtwo$ form $\varphi$ is coclosed, that is 
$d\psi=0$. 

Now, let $G$ be a $7$-dimensional simply connected nilpotent
Lie group with Lie algebra $\mathfrak g$. Then, a $\Gtwo$-structure 
on $G$ is \emph{left invariant} if and only if the corresponding
$3$-form $\varphi$ is left invariant. Thus, a left invariant $\Gtwo$-structure on 
$G$ corresponds to an element $\varphi$ of $\Lambda^3({\mathfrak g}^*)$ that 
can be written as 
\begin{equation}\label{eqn:3-forma G2}
 \varphi=e^{127}+e^{347}+e^{567}+e^{135}-e^{146} -e^{236}-e^{245},
\end{equation}
with respect to some orthonormal coframe $\{e^1,\dotsc, e^7\}$ of the dual space ${\mathfrak{g}}^*$,
where $e^{127}$ stands for $e^1\wedge e^2\wedge e^7$, and so on.
So the dual form $\psi=\star_{\varphi}\varphi$
has the following expression
\begin{equation}\label{eqn:4-forma G2}
 \psi=e^{1234}+e^{1256}+e^{1367}+e^{1457}+e^{2357}  -e^{2467} +e^{3456}.
\end{equation}

Note that in order to recover the left invariant $\Gtwo$ form $\varphi$
from the 4-form $\star_{\varphi} \varphi$ we need  to fix an orientation of $\mathfrak g$.
In fact, the stabilizer of $\star_{\varphi} \varphi$ in $\mathrm{GL}(7,\mathbb{R})$ is $\Gtwo \times \mathbb{Z}_{2}$ since the matrix
$-Id$ preserves the form $\star_{\varphi}\varphi$,  and so the latter fails to determine the overall orientation.

Recall that the  seven dimensional Heisenberg group $H$
is the simply connected nilpotent Lie group whose Lie algebra $\mathfrak h$ is defined by
\begin{equation}\label{eq:lieheisenberg}
\frak h=\left (0,0,0,0,0,0,\frac{\sqrt{6}}{6}(e^{12}+e^{34}+e^{56})\right ).
\end{equation}
This notation means that the dual space $\frak {h}^*$ is spanned by $\{ e^1,\ldots ,e^7\}$ satisfying
$$
de^i=0, \quad 1\leq i\leq 6, \qquad de^7=\frac{\sqrt{6}}{6}(e^1 \wedge e^2 +e^3 \wedge e^4 +e^5 \wedge e^6).
$$

\section{On the coflows of cococlosed $\Gtwo$-structures}\label{section-coflows}

Here we show  the expression of each one of the coflows \eqref{co-flow} and \eqref{modifiedcoflow} in terms
of the intrinsic torsion forms of a coclosed $\Gtwo$-structure  \cite{Bryant, Grigorian}.

Let $M$ be a $7$-dimensional manifold with a $\Gtwo$-structure defined by a $3$-form $\varphi$.   Denote by $\psi$ the 4-form $\psi = \star_{\varphi} \varphi$, where 
$\star_{\varphi}$ is the Hodge star operator of the metric $g_{\varphi}$ induced by $\varphi$. Let $(\Omega^* (M), d)$ be the de Rham complex of differential forms on $M$. Then, Bryant  in \cite{Bryant} proved that the 
forms $d \varphi$ and $d \psi$ are such that 
\begin{gather}\label{torsion}
\begin{cases}
d\varphi= \tau_0\,\psi+3\,\tau_1\wedge\varphi+\star_{\varphi} \tau_3,\\
d\psi=4\tau_1\wedge\psi-\star_{\varphi} \tau_2,
\end{cases}
\end{gather}
where $\tau_0\in\Omega^0 (M), \tau_1\in\Omega^1 (M), \tau_{2}\in\Omega_{14}^2(M)$ and $\tau_3\in\Omega_{27}^3(M)$. Here
$\Omega_{14}^2(M)$ and $\Omega_{27}^3(M)$ are the spaces
\begin{equation*}
\begin{array}{l}
\Omega_{14}^2(M)=  \{ \alpha \in \Omega^2 (M) \, \mid  \, \alpha \wedge \varphi = - \star_{\varphi} \alpha \},\\.
\end{array}
\end{equation*}
\begin{equation*}
\begin{array}{l}
\Omega^3_{27}  (M)= \{ \beta \in \Omega^3 (M) \, \mid  \, \beta \wedge \varphi =0= \beta \wedge \star_{\varphi}  \varphi \}.
\end{array}
\end{equation*}
The differential  forms $\tau_i$ ($i=0, 1, 2, 3$) that appear in \eqref{torsion}, are called the {\em intrinsic torsion forms} of $\varphi$.
According to Grigorian \cite{Grigorian}  the {\em full torsion tensor} $\tau$ of $\varphi$ is the  tensor  field on $M$  given by
$$
\tau=\frac{1}{4}\tau_0\,g_{\varphi}-\imath_{\tau_1}\varphi-\frac{1}{3}j_{\varphi}(\tau_3)+\frac{1}{2}\tau_{2},
$$
where  
$\imath_{\tau_1}$ denotes the contraction by $\tau_1$ using the metric $g_{\varphi}$  induced by $\varphi$
(that is, if $U$ is the vector field on $M$ such that $\tau_1\,=\,\imath_{U} g$, then $\imath_{\tau_1}\varphi=\imath_{U}\varphi$)
and $j_{\varphi}\,\colon\, \Omega^3 (M)\,\longrightarrow\, S^{2}(M)$ is the map defined by
$$
j_{\varphi}(\gamma)(X, Y)\,=\,\star_{\varphi} \Big((\imath_{X} \varphi ) \wedge(\imath_{Y} \varphi )  \wedge \gamma\Big),
$$
where $\gamma\in \Omega^3 (M)$, and $X, Y$ are vector fields on $M$ \cite{Bryant}.  
In particular, by \cite{Bryant} $j_{\varphi}$ 
is an isomorphism between the space  
$\Omega^3_{27}  (M)$
and  the space  $S^2_0 (M)$ of  trace-free  symmetric $2$-tensors on $M$.

Recall that $\varphi$ defines a coclosed $\Gtwo$-structure on $M$ if $\psi$ is closed, that is $d \psi =0$.  
In this case, \eqref{torsion} implies that  the forms $\tau_1$ and $\tau_{2}$ vanish, and so 
the full torsion tensor $\tau$ has the following expression  
$$
\tau=\frac{1}{4} \tau_0\,g_{\varphi} -\frac{1}{3}j_{\varphi}(\tau_{3}).
$$ 
Since $\tau_3  \in \Omega^3_{27}(M)$,   the trace  of  $j_{\varphi}(\tau_{3})$ vanishes. Therefore, $\mathrm{Tr}(\tau)$ of $\tau$ is given by 
\begin{equation}\label{exp:trace}
\mathrm{Tr}(\tau)= \frac{1}{4} \tau_0\,  \mathrm{Tr}(g_{\varphi}) =    \frac{7}{4}\tau_0.
\end{equation}

\medskip

\begin{lemma} \label{lem:relations-torsion}
Let $M$ be a $7$-dimensional manifold with a coclosed  $\Gtwo$ form $\varphi$. Denote
by $\tau_0$ and  $\tau_3$ the torsion forms of $\varphi$. Then,  the torsion forms $\widetilde{\tau_0}$ and  $\widetilde{\tau_3}$ of $-\varphi$ satisfy
\begin{equation}\label{relations-torsion}
\widetilde{\tau_0}=-\tau_0, \quad \widetilde{\tau_3}=\tau_3.
\end{equation}
\end{lemma}
\begin{proof}
Using \eqref{torsion}, we see that $\widetilde{\tau_0}=-\tau_0$ and $\widetilde{\tau_3}=\tau_3$ since $\star_{-\varphi}=-\star_{\varphi}$.
\end{proof}

\medskip

\begin{proposition} \label{prop:exp-C-G}
 Let $M$ be a $7$-dimensional manifold with a coclosed  $\Gtwo$ form $\varphi$. Then, the coflow \eqref{co-flow} for $\varphi$ has the following expression
\begin{equation*}
\begin{aligned}
\mathrm{(C)}\quad  &\frac {\partial}{\partial t} \psi(t)= 
- d\big(\tau_0 (t)\big)\wedge\varphi (t)-\big(\tau_0(t)\big)^2 \psi(t) -\tau_0(t)\star_t \tau_3(t)-d\tau_3(t),\\ &
d \psi(t)= 0, \quad \varphi(0)= \varphi,
\end{aligned}
\end{equation*}
and the  modified coflow \eqref{modifiedcoflow} is expressed as 
\begin{equation*}
\begin{aligned}
\mathrm{(G)}\quad  &\frac {\partial}{\partial t} \psi(t)=
\tau_0  (t)\left(2A-\frac{5}{2} \tau_0 (t) \right)\psi(t)+\left(2A-\frac{5}{2} \tau_0 (t)\right)*_t \tau_3 (t)+d \tau_3 (t)\\
&\quad  \quad \quad  +\frac{5}{2}\varphi (t) \wedge d \tau_0(t),\\
& d \psi(t) =0, \quad  \varphi(0)= \varphi,
\end{aligned}
\end{equation*}
where $\tau_{0}(t)$ and  $\tau_{3}(t)$ are the torsion forms of $\varphi(t)$ {\rm(}according with \eqref{torsion}{\rm)}, $\star_{t}$ is the Hodge star operator with respect to the Riemannian metric $g_{\varphi(t)}$ induced by $\varphi(t)$  and $A$ is a fixed positive constant.
\end{proposition}
\begin{proof}
Since the solution $\psi (t)$ to the coflow \eqref{co-flow}, if it exists, remains closed and    \eqref{modifiedcoflow} preserves  the  closedness of $\psi (t)= \star_{t} \varphi(t)$, by \eqref{torsion} and the vanishing of the torsion forms $\tau_1(t)$ and $\tau_2(t)$ of $\varphi(t)$,
$$
d \varphi(t) = \tau_0(t) \psi(t) + \star_t \tau_3(t).
$$
 Hence,
\begin{eqnarray*}
\Delta_{t} \psi (t)&=&d\,d^*\psi (t)=d \star_t d\varphi (t)=d \star_ t \Big(\tau_0 (t)\psi (t) +\star_t \tau_3 (t)\Big)\\
&=& d(\tau_0 (t))\wedge\varphi (t)+\tau_0 (t)^2\psi (t)+\tau_0 (t) \star_t \tau_3 (t)+d\tau_3 (t),
\end{eqnarray*}
and
\begin{eqnarray*}
2d \Big(\big(A-\mathrm{Tr}(\tau (t))\big) \varphi (t)\Big) &=&2d\left( (A-\frac{7}{4}\tau_0 (t)) \varphi (t) \right)\\
&=&-\frac{7}{2}d(\tau_0 (t))\wedge\varphi (t)+\left(2A-\frac{7}{2}\tau_0 (t) \right)(\tau_0 (t) \psi (t) + \star_t \tau_3 (t)).
\end{eqnarray*}
Thus,
\begin{eqnarray*}
\Delta_{t}\psi+ 2d \Big(\big(A-\mathrm{Tr}(\tau (t))\big) \varphi (t)\Big) &=&
-\frac{5}{2}d(\tau_0 (t))\wedge\varphi (t)+\tau_0 (t)\left(2A-\frac{5}{2}\tau_0 (t)\right)\psi (t)\\
&&+\left(2A-\frac{5}{2}\tau_0 (t)\right) \star_t \tau_3 (t)+d\tau_3 (t),
\end{eqnarray*}
and the Proposition follows.
\end{proof}

\begin{remark} \label{rem:gen-opposite}
Note that \eqref{relations-torsion} and Proposition \ref{prop:exp-C-G} imply that the solution of the coflow (G) for $\varphi$
(if such a solution exists) changes when the initial coclosed $\Gtwo$ form is $-\varphi$ instead of $\varphi$  (see
Theorem \ref{Th6.1} and Theorem \ref{Th6.1-2}). However, the study of the coflow (C) 
is independent of whether the initial condition is $\varphi$ or $-\varphi$.
\end{remark}

\begin{remark}    
By \cite{Grigorian}, since $ \mathrm{Tr}(\tau (t))  = \frac 7 4  \tau_0 (t)$,  as long as the condition $0 \leq \frac 7 4  \tau_0 (t)  \leq \frac 43 A$ holds 
 for the time of  existence, we have the following inequality  for the volume  
 $$
A   \int_M   \frac{7}{4}\,\tau_0 (t)\,\mathrm{vol}  \geq    \int_M  \frac {3}{4} \Big(\frac{7}{4}\,\tau_0 (t)\Big)^{2}\, \mathrm{vol}.
$$ 
\end{remark}

\medskip

\section{Explicit solutions for the Laplacian coflow}\label{sect:coflow}

In this section we study the Laplacian coflow on the seven dimensional Heisenberg Lie group $H$ with structure equations \eqref{eq:lieheisenberg}.\par\bigskip
Let   $\varphi_0$  be a  left invariant coclosed $\Gtwo$-structure on $H$. Denote  by $g_0$ the underling metric  and  by  $ \psi_0 = \star_0 \varphi_0$  its  Hodge dual.  

Let   $\eta = \| e^7  \|_0^{-1} e^7$. Clearly  $\|\eta \|_0 = 1$ and $d\eta\in\Lambda^2\mathrm{Ker}(\eta)^*$  is a non-degenerate two-form on $\mathrm{Ker}(\eta)$. Moreover, $ \mathrm{Ker}(\eta)^* = {\mbox Span}\la  e^1, \ldots, e^6\ra$ and the $1$-forms $e^j$, $j = 1, \ldots, 6$, are all closed.  If we  identify $\frak h$  with $\mathrm{Ker}(\eta)\oplus \frak z$, being  
 $\frak z=[\frak h, \frak h] = {\mbox Span}\la e_7\ra$
the commutator of $\frak h$,  then every  four-form $\psi  \in \Lambda^4 \frak h^*$ has a unique decomposition as
\begin{equation}\label{decomp}
\psi=\psi^{(4)}+\psi^{(3)}\wedge\eta,
\end{equation}
where $\psi^{(i)}\in\Lambda^i\mathrm{Ker}(\eta)^*,\;i=3,4$,  are closed forms.\par

Denote by $\star_0$ and $*_0$ the Hodge operators on $\frak h$ and $\mathrm{Ker}(\eta),$ respectively. 
Note that  the  $\Gtwo$-structure $\varphi_0$ defines an $\SU(3)$-structure
 $(\omega_0, \rho_0)$ on  $\mathrm{Ker}(\eta)$. Using this fact, the four-form $ \psi_0 = \star_0 \varphi_0$ on $\frak h$ can be written as
$$
\psi_0 = \frac{1}{2}\omega_0^2+  \widehat  \rho_0 \wedge \eta,
$$
where $ \widehat \rho_0 =  J_0 \rho_0$, and  $J_0$  is the almost complex structure induced by  $(\omega_0, \rho_0)$. Indeed, 
if $x_0 \in \frak h$  is  the vector  defined by
$$
g_0(x_0,y)=\eta(y), 
$$
for every $y\in \frak h$, then
$$
\mathrm{Ker}(\eta)=\left\{y\in\ {\frak h}\,|\,g_0 (x_0,y)=0\right\} ={\mbox {Span}}\la x_0\ra ^{{\perp}_{0}},
$$
and  we can apply Proposition 4.5 in \cite{Schulte} to define the $\SU(3)$-structure $(\omega_0, \rho_0)$. 

For a general $\SU(3)$-structure on a real vector space we 
 have the following result.

\begin{lemma}\label{lem}
Let $(\omega,\rho)$ be a linear $\SU(3)-$structure on $\R^6$, and let $\alpha\in\Lambda^2(\R^6)^*$. Then the following inequalities hold
\begin{enumerate}
\item[1.] $\| \alpha \|^2+ \| \frac 12 \omega^2\wedge\alpha\|^2= \| \alpha\wedge\omega\|^2\leq 4 \|\alpha\|^2;$
\smallskip

\item[2.]  $\|\alpha^3\|^2\leq 6  \| \alpha\|^6$,  where $\| \cdot \|$ is the norm  induced by  the scalar product   
defined by the $\SU(3)$-structure $(\omega,\rho)$.
\end{enumerate}  
\end{lemma}
\begin{proof}
Let us fix an orthonormal basis 
$\{ e^1, \ldots, e^6 \}$ of $(\R^6)^*$ so  that $\omega=e^{12}+e^{34}+e^{56}$, and write 
$\alpha=\sum_{1\leq h< k \leq 6}a_{hk}e^{hk}$. 
 Then,
\begin{equation}\label{eqn:norm1}
\| \alpha \|^2= \sum_{1\leq h<h \leq 6}^6 a_{hk}^2.
\end{equation}
On the other hand, 
$$
\begin{array}{lcl}
\omega\wedge\alpha&=&e^{12} \wedge \left(a_{34}e^{34}+a_{35}e^{35}+a_{36}e^{36}+a_{45}e^{45}+a_{46}e^{46}+a_{56}e^{56}\right)\\
&& + e^{34} \wedge \left(a_{12}e^{12}+a_{15}e^{15}+a_{16}e^{16}+a_{25}e^{25}+a_{26}e^{26}+a_{56}e^{56}\right)\\
&& + e^{56} \wedge \left(a_{12}e^{12}+a_{13}e^{13}+a_{14}e^{14}+a_{23}e^{23}+a_{24}e^{24}+a_{34}e^{34}\right).\\
\end{array}
$$
Thus,
\begin{equation} \label{eqn:norm2}
\| \omega\wedge\alpha \|^2=  \| \alpha \|^2+\left(a_{12}+a_{34}+a_{56}\right)^2= \| \alpha \|^2+  \left \| \frac 12 \, \omega^2\wedge\alpha \right  \|^2.
\end{equation}
 Moreover,
$$ \left \| \frac 12  \omega^2\wedge\alpha   \right \|^2= \| *(\omega)\wedge\alpha\|^2=\left(\omega|\alpha\right)^2\leq \| \omega \|^2 \| \alpha\|^2=3 \|\alpha \|^2.$$
\par
This equality together with \eqref{eqn:norm1} and \eqref{eqn:norm2} imply the first part of the Lemma.

To prove $2.$ note that the spectral theorem guarantees the existence of an orthonormal  basis of $1$-forms   $\{ f^1, \ldots, f^6 \}$  such  that 
$\alpha=\lambda_1f^{12}+\lambda_2f^{34}+\lambda_3f^{56}$, for some real numbers $\lambda_i$ with $i=1, 2,3$.    
Indeed, any real skew-symmetric matrix can be diagonalized by a unitary matrix. Since the eigenvalues of a real skew-symmetric matrix are imaginary, it is possible to transform it  to a block diagonal form by an  orthogonal transformation. 
Therefore,
$$
\| \alpha \|^2=\lambda^2_1+\lambda_2^2+\lambda_3^2,
$$
 and
$$ 
{\alpha^3}= 6 \lambda_1\lambda_2\lambda_3f^{123456}.
$$
Thus,
$$
\|  \alpha^3 \|^2= 36 \lambda_1^2\lambda_2^2\lambda_3^2,
 $$
and 2. follows.
\end{proof}

\begin{theorem}\label{Hflow} Let $H$ be the seven dimensional Heisenberg group whose Lie algebra is defined by \eqref{eq:lieheisenberg}. 
Then, for any left invariant coclosed $\Gtwo$ form  $\varphi_0$,  the solution $\phi_t$ of the Laplacian coflow \eqref{co-flow}
 with initial condition $\psi_0 = \star_0 \varphi_0$ is given by
$$\psi(t)=\frac{1}{2}\omega (t)^2+  \widehat \rho (t) \wedge\frac{1}{\varepsilon_t}\,\eta,$$
where $6\,\varepsilon_t^2=*_0(\omega(t)^3)$, and 
 $\omega (t)$ and $\widehat \rho (t)$
are forms on $\mathrm{Ker}(\eta)$,  given  respectively by
\begin{gather*}
\omega (t) = \lambda_1(t)f^{12}+\lambda_2(t)f^{34}+\lambda_3(t)f^{56},\\[3pt]
\widehat \rho(t)=\sqrt{\lambda_1(t)\lambda_2(t)\lambda_3(t)}\left(-f^{246}+f^{136}+f^{145}+f^{235}\right),\\
\end{gather*}
with respect to some   $g_0$-orthonormal frame $\{ f_1, \ldots, f_6 \}$ of $\mathrm{Ker}(\eta)$, and  the functions $\lambda_i (t)$, $i = 1,2,3$,  satisfy 
\begin{equation}\label{flow2}
\left \{  \begin{array}{l} \lambda'_{1}(t)=-\frac{\lambda_2(t)\lambda_3(t)+n_{2}n_{3}\lambda_1^2(t)}{\lambda_1(t)^2\lambda_2(t)^2\lambda_3(t)^2},\\ [4pt]
\lambda'_{2}(t)=-\frac{\lambda_1(t)\lambda_3(t)+n_{1}n_{3}\lambda_2^2(t)}{\lambda_1(t)^2\lambda_2(t)^2\lambda_3(t)^2},\\[4pt]
\lambda'_{3}(t)=-\frac{\lambda_1(t)\lambda_2(t)+n_{1}n_{2}\lambda_3^2(t)}{\lambda_1(t)^2\lambda_2(t)^2\lambda_3(t)^2},\\[4pt]
\lambda_1(0) =   \omega(0) (f_1, f_2), \,\,  \lambda_2(0) =   \omega(0) (f_3, f_4), \, \, \lambda_3(0) =  \omega(0) (f_5, f_6),
\end{array}
\right.
\end{equation}
 for  $n_j \in \left\{1,-1\right\}$. 
In particular,   the solution  is ancient with singular time $0<T< \frac{4}{\sqrt[3]{6\,||d\eta||_0^2}}$.
\end{theorem}
\begin{proof}
We are going to show that  the system \eqref{co-flow} turns out to be equivalent to the system of ODEs given by \eqref{flow2}. But first let us observe that the initial $\psi_0 = \star_0 \varphi_0$ is $H$-invariant and the system \eqref{co-flow} is invariant by diffeomorphisms, whence $H$-invariant too, and therefore the system \eqref{co-flow} reduces to a system of  ODEs on $\Lambda^4 \frak h^*$.  This ensures the existence of a unique $H$-invariant solution $\psi_t$ of \eqref{co-flow} for short times. Now let  $\varepsilon_t$ be the norm $\| \eta \|_t$ of $\eta$ with respect to the metric induced by $\psi(t)$.
We can write
$$\psi(t)=\frac{1}{2}\omega(t)^2 + \widehat \rho (t) \wedge\frac{1}{\varepsilon_t}\eta,$$
where  the pair  $(\omega(t),\rho(t))$ defines an  $\SU(3)-$structure on $\mathrm{Ker}(\eta)$ and $\widehat \rho (t) = J_t \rho(t)$. 
 In fact, if $x_t \in \frak h$ is the vector defined by $g_t(x_t,y)=\eta(y)$, for any $y\in \frak h$,
then $\mathrm{Ker}(\eta)=\left\{y\in\ {\frak h}\,|\,g_t(x_t,y)=0\right\}$ is the orthogonal complement of the span of $x_t$ with respect to $g_t$. Thus we can apply Proposition 4.5 in \cite{Schulte}.\par
With respect to  the decomposition \eqref{decomp} we have 
$$
\psi (t) =\psi^{(4)} (t) +\psi^{(3)} (t) \wedge\eta,
$$
so,
$$\psi^{(4)} (t)=\frac{1}{2}\omega(t)^2,\quad \psi^{(3)} (t)=\frac{1}{\varepsilon_t} \widehat \rho(t).$$
 Moreover,  the forms $\omega(t) \in \Lambda^2 \mathrm{Ker}(\eta)^*$ and  $\widehat \rho(t) \in \Lambda^3 \mathrm{Ker}(\eta)^*$ are closed.  
Since  $\frac{d}{dt}{\psi}(t)$  
 is exact, the cohomology class of $\psi(t)$ is fixed by the flow, and  hence
$$\frac{d}{dt}  {\psi}(t)= \frac{d}{dt}{\psi}^{(4)}(t)+ \frac{d}{dt}{\psi}^{(3)} (t) \wedge \eta\in d\Lambda^3{\frak h}^*\subseteq \Lambda^4\mathrm{Ker}(\eta)^*.$$
Therefore, $ \frac{d}{dt}{\psi}^{(3)} (t)=0$ and  $$ \widehat \rho (t)=\varepsilon_t\psi^{(3)} (0)=\varepsilon_t \widehat \rho_0.$$
Consequently, the  almost complex structure $J_t$ defined by $\rho(t)$  does not change along the flow, i.e. $J_t \equiv J_0$, where $J_0$ is the almost complex structure defined by $\rho_0$. Thus $\rho(t)=-J_0 \widehat  \rho (t) =\varepsilon_t \rho_0$  and
\begin{equation} \label{relationvarepsilon}  \frac{1}{6}\omega(t)^3=\frac{1}{4}\rho(t) \wedge\widehat \rho(t)=\varepsilon^2_t*_0(1), \end{equation}
where in the first equality we used the fact that $(\omega(t),\rho(t))$ defines an $\SU(3)$-structure on $\mathrm{Ker}(\eta)$.
 \par
Now let us compute the Laplacian of $\psi(t)$ with respect to the metric $g_t$:
$$
\begin{array}{lcl}
d\star_td \star_t \psi (t)&=&d\star_td \star_t(\frac{1}{2}\omega(t)^2 +\ \widehat \rho(t)\wedge \varepsilon^{-1}_t\eta)\\
&=&d\star_td(\omega (t) \wedge\varepsilon^{-1}_t\eta+ \rho(t))\\
&=&d\star_t(\varepsilon^{-1}_t\omega(t)\wedge d\eta)\\
&=&d\left(\varepsilon^{-2}_t*_t(\omega(t)\wedge d\eta)\wedge\eta\right)\\
&=&\varepsilon^{-2}_t *_t(\omega(t)\wedge d\eta)\wedge d\eta.
\end{array}
$$
On the other hand we have
$$\frac{d}{dt}\psi_t=\frac{d}{dt} \left ( \frac{1}{2}\omega(t)^2 \right ).$$
Thus, by $\frac{d} {d t} \psi(t) = -  \Delta_{t} \psi(t)$ we obtain
\begin{equation} \label{coflownew2}
\frac{1}{2}\left(\frac{d}{dt}\omega(t)^2\right)=-\varepsilon^{-2}_t *_t(\omega(t) \wedge d\eta)\wedge d\eta.\end{equation}
We observe that, being $d\psi(t)=0$,
$$0=\varepsilon_td \psi (t)=d\left(\varepsilon_t\frac{1}{2}\omega(t)^2+\hat{\rho}(t)\wedge\eta\right)= \widehat \rho (t) \wedge d\eta.$$ 
Since $\hat \rho$ is the imaginary part of a $(3,0)$-form,  $\eta$ must be of type $(1,1)$ and hence it is $J_0$-invariant,  i.e. $J_0 (d \eta) = d \eta$.

Fixing a frame $\left(x_1,\dots,x_6\right)$ of $\mathrm{Ker}(\eta)$ \and using $*_t \omega(t) = \frac{1}{2} \omega(t)^2$  we get
\begin{equation*}
\begin{array}{lcl}
*_t(d\eta\wedge\omega(t))&=& \sum_{1\leq i<j\leq 6}*_t((d\eta)_{ij}x^{ij}\wedge \omega(t))
\\
&=&- \sum_{1\leq i<j\leq 6}(d\eta)^{ij}x_i \lrcorner  x_j \lrcorner  *_t\omega(t)\\
&=&- \sum_{1\leq i<j\leq 6}(d\eta)^{ij}x_i \lrcorner x_j \lrcorner \frac 12 \omega(t)^2\\
&=&- \sum_{1\leq i<j\leq 6}(d\eta)^{ij}x_i  \lrcorner  ((x_j \lrcorner \omega(t))\wedge\omega(t)).
\end{array}
\end{equation*}
Moreover
$$
\begin{array}{lcl}
x_i\lrcorner((x_j\lrcorner\omega(t))\wedge\omega(t))&=&\omega (t) (x_j,x_i)\omega (t) -(x_j\lrcorner\omega(t))\wedge(x_i\lrcorner\omega(t))\\
&=&-\omega(t)_{ij}\omega(t)+(x_i\lrcorner\omega(t))\wedge(x_j\lrcorner\omega(t)).
\end{array}$$
Now, taking into account the fundamental relation $\omega(t) (x, y) = g_t ( x, J_0y)=-[(J_0)^*(x\lrcorner g)](y)$, we have
\begin{equation*}
x_i\lrcorner\omega(t)=-(J_0)^*\left(\sum_mg_{im}(t)x^m\right),\quad x_j\lrcorner\omega(t)=-(J_0)^*\left(\sum_ng_{jn}(t)x^n\right).
\end{equation*}
Therefore,
\begin{equation*}
\begin{array}{lcl}
*_t(d\eta\wedge\omega(t))&=& \sum_{1\leq i<j\leq 6}(d\eta)^{ij}\left[\omega(t)_{ij}\, \omega(t)-x_i \lrcorner \omega(t)\wedge x_j \lrcorner \omega(t)\right]\\
&=&\sum_{1\leq i<j\leq 6}(d\eta)^{ij}\left[\omega_{ij}(t)\,  \omega(t) -\sum_{m,n=1}^6 g_{im}(J_0^*x^m) \wedge g_{jn}(J_0^*x^n)\right]\\
&=& \sum_{1\leq i<j\leq 6}(d\eta)^{ij} \omega_{ij}(t)\,  \omega(t) -\sum_{m,n=1}^6(d\eta)_{mn}J_0^* x^m \wedge J_0^*x^n\\
&=&(d\eta, \omega(t))_t \, \omega(t) -J_0^* (d\eta)\\
&=&(d\eta, \omega(t))_t \, \omega(t) - d\eta,
\end{array}
\end{equation*}
 where $( \cdot, \cdot )_t$ denotes the scalar product induced by  $g_t$ and $\lrcorner$ is the contraction.

Therefore we can reformulate \eqref{coflownew2} as  
\begin{equation}\label{flow}
\frac{d}{dt}\left( \frac{1}{2} \omega(t)^2 \right)=-\varepsilon^{-2}_t\left[(d\eta, \omega(t))_t \, \omega (t) \wedge d\eta-d\eta \wedge d \eta \right],
\end{equation}\par
with $\omega (t)\in\Lambda^2\mathrm{Ker}(\eta)^*.$
Define on $\mathrm{Ker}(\eta)$  the  following bilinear form 
$$h(x,y)=d\eta(x,J_0y),\quad x,y\in\mathrm{Ker}(\eta).$$
Since $J_0 (d\eta)=d\eta$,  we have that $h$  is symmetric. If we consider $\mathrm{Ker}(\eta)$ as complex vector space through $J_0$,
 then both $g_0$ and $h$ define  
on $\mathrm{Ker}(\eta)$ sesquilinear forms $g^c$ and $h^c$, respectively. Clearly $g^c$ is positive definite  while $h^c$ is non-degenerate with  mixed signature. By the spectral theorem there exists a complex $g^c-$orthonormal basis $\left\{k_1,k_2,k_3\right\}$ of  the complex vector space $\mathrm{Ker}(\eta)$  such that
$$h^c(k_i,k_j)=\delta_{ij}\frac{n_i}{l_i},\quad n_i\in\left\{1,-1\right\},\;l_i> 0.$$
Therefore,  if  $\left\{k^1,k^2,k^3\right\}$ is the dual basis of   $\left\{k_1,k_2,k_3\right\}$, 
putting $f^i=\sqrt{l_i} \, k^i$, for $i=1,2,3$, we get 
$$d\eta=\sum_{i=1}^3 n_if ^i\wedge J_0 f^i,\quad \omega_0=\sum_{i=1}^3 l_i f^i\wedge J_0 f^i.$$ \par
In order to find  $\omega(t)$, let us suppose that it is given by 
\begin{equation} \label{expressioneomega_t} \omega(t)=\sum_{i=1}^3 \lambda_i(t) f^i\wedge J_0  f^i,
 \end{equation}
where  $\lambda_1(t),\lambda_2(t)$ and $\lambda_3(t)$  are  positive functions such that   
$\lambda_j(0) = l_j$. Then  $\omega(t)$ is a  non-degenerate $2$-form of type $(1,1)$ with respect to 
 $J_0$, and
$$
\frac{d}{dt} \left (  \frac 12 \omega (t)^2  \right ) =\sum_{i<j}\left\{{\lambda'_i}(t)\lambda_j(t)+\lambda_i(t){\lambda'_j}(t)\right\}f^i\wedge J_0 f^i\wedge f^j\wedge J_0 f^j.
$$  
Using \eqref{relationvarepsilon} and \eqref{expressioneomega_t}  we have 
$$\varepsilon_t^2= \frac 16 *_0\omega(t)^3 =\lambda_1(t)\lambda_2(t)\lambda_3(t).$$  Now, from the equation
\eqref{flow}  and $h^c(f_i, f_j) =  \delta_{ij} n_i$  we get
\begin{align*}
&\sum_{i<j}\left\{{\lambda'_i}(t)\lambda_j(t)+\lambda_i(t){\lambda'_j}(t)\right\}f^i\wedge J_0 f^i\wedge f^j\wedge J_0 f^j =\\
&-\frac{1}{\varepsilon^2_t}\sum_{i<j}\left[\left(\frac{n_1}{\lambda_1(t)}+\frac{n_2}{\lambda_2(t)}+\frac{n_3}{\lambda_3(t)}\right)\left(n_i\lambda_j(t)+n_j\lambda_i(t))\right)-2n_in_j\right]f^i\wedge J_0 f^i\wedge f^j\wedge J_0 f^j.\\
\end{align*}
This is the system of ordinary differential equations given by  \eqref{flow2}.
This system does indeed have a unique solution, which, in turn,
defines a non-degenerate 2-form compatible with $J_0$ by \eqref{expressioneomega_t}. 
Such a form has to satisfy \eqref{flow} by construction. Therefore, we find out that the solution of \eqref{co-flow} is given by 
$$
\psi(t)=\frac{1}{2}\omega (t)^2+  \widehat \rho (t) \wedge\frac{1}{\varepsilon_t}\,\eta,
$$
where   $\widehat \rho (t) = \varepsilon_t\widehat{\rho}(0)$ and  $\omega(t)$ as in   \eqref{expressioneomega_t}, with the   functions $\lambda_i (t)$, $i = 1,2,3$,  solving \eqref{flow2}.
\par
 Let $(\tau,T)$ be the maximal interval of existence of $\psi(t)$, where $-\tau,T\in[-\infty,+\infty]$. We want to prove that $T<+\infty$ and $\tau=-\infty$. Computing the derivative  of  $6\,\varepsilon_t^2=*_0(\omega(t)^3)$  we get  $$12 \, \varepsilon_t \varepsilon'_t=3*_0 \left ( \frac{d}{dt} {\omega}(t) \wedge\omega(t)^2 \right ).$$ 
Then,
\begin{equation*}
\begin{array}{lcl}
\varepsilon'_t&=&-\frac{\varepsilon^2_0}{4\varepsilon^{3}_t}*_0\left[*_t\left(\omega(t)\wedge d\eta\right)\wedge \omega(t) \wedge d\eta\right]\\[4pt]
&=&-\frac{\varepsilon^2_0}{4\varepsilon^{3}_t}*_0(\| \omega(t)\wedge d\eta \|^2_t*_t(1))\\[4pt]
&=&-\frac{1}{4\varepsilon^{3}_t}*_0 (\| \omega(t)\wedge d\eta \|^2_t \, \varepsilon^2_t \, *_0(1))\\[4pt]
&=&-\frac{1}{4\varepsilon_t} \| \omega(t)\wedge d\eta \|^2_t.\\
\end{array}
\end{equation*}
This implies that $\varepsilon_t'<0$ and also
 the existence of $\lim_{t\rightarrow T}\varepsilon_t=\varepsilon_T\geq 0$.   
Note that $\varepsilon_t^{-1} \eta$ is the unit vector orthogonal to ${\mbox{Ker}}(\eta)$
such that  $\star_t(1)=*_t(1)\wedge\left(\varepsilon_t^{-1}\eta\right)$, thus\begin{equation}\label{stardiuno} *_t (1) =  \star_t  \left (\frac{1}{\varepsilon_t}\eta \right ),  \end{equation} where $\star_t$ is the  Hodge star operator with respect to the metric induced by $\psi(t)$.   Moreover, the volume form $*_0 (1)$ is proportional to the $6$-form $d \eta^3$, since both are non-zero $6$-forms on ${\mbox{Ker}}(\eta)$.
Therefore,  we can write 
\begin{equation} \label{exprvoldeta}
 *_0 (1) = \frac{1}{ 6 \delta_0} (d \eta)^3,
\end{equation}
 where 
 $$
\delta_0=  \frac 16 \|  (d\eta)^3 \|_0.
$$
Thus,  using  \eqref{stardiuno}, \eqref{exprvoldeta} and 
  $*_t (1) = \frac 16 \omega(t)^3 = \varepsilon_t^2 *_0 (1)$, we obtain
$$\star_t  \left  (\frac{1}{\varepsilon_t}\eta \right )=*_t(1)={\varepsilon_t^2}*_0(1)= \frac{1}{6}  \frac{\varepsilon_t^2}{\delta_0}(d\eta)^3,
$$
that is
$$
\star_t   \left(  \varepsilon_t^{-1} \eta \right) = \frac{1}{6}\frac{\varepsilon_t^2}{\delta_0} d \eta^3.
$$
Taking the square norm of the previous expression gives
$$
1=\|  \star_t   \left(  \varepsilon_t^{-1} \eta \right)  \|_t^2  = \left\|\frac{1}{6}   \frac{\varepsilon_t^2}{\delta_0} d \eta^3 \right\|_t^2,
$$
whence
$$
\| d\eta^3 \|_t^2 =  36 \frac{\delta_0^2}{\varepsilon_t^4}.
$$
 Now by  Lemma \ref{lem}  we  have  
 $$ 36 \frac{\delta_0^2}{\varepsilon_t^4} = \| d \eta^3 \|_t^2 \leq 6 \| d \eta \|_t^6,$$
and
$$ \| \omega (t) \wedge d\eta \|^2_t\geq  \| d\eta \|^2_t.$$
Therefore, we can estimate $\varepsilon'_t$ as follows
$$\varepsilon'_t=-\frac{1}{4\varepsilon_t} \| \omega(t) \wedge d\eta \|^2_t\leq -\frac{1}{4\varepsilon_t} \| d\eta||_t^2\leq -\frac{1}{4\varepsilon_t}\left(\frac 16    \| (d\eta)^3 \|^2_t\right)^{\frac{1}{3}}=-\frac{\sqrt[3]{6\,\delta_0^2}}{4\varepsilon_t\sqrt[3]{\varepsilon_t^4}}=-\frac{C}{\varepsilon_t^{1+\frac{4}{3}}},$$
with $C=\frac{\sqrt[3]{6\,\delta_0^2}}{4}$.
As a consequence, if $t\in(0,T)$, we get 
$$\varepsilon_t-1=\int_0^t\varepsilon'_sds\leq-C\int_{0}^t\frac{1}{\varepsilon_s^{1+\frac{4}{3}}}ds\;\Rightarrow\;t=\int_{0}^tds\leq\int_{0}^t\frac{1}{\varepsilon_s^{1+\frac{4}{3}}}ds\leq\frac{1-\varepsilon_t}{C},$$
where we have used that $\varepsilon_s<1$ if $s\in(0,t)$.  So 
$$T\leq\frac{1-\varepsilon_T}{C}=4\frac{1-\varepsilon_T}{\sqrt[3]{6\,\delta_0^2}}\leq \frac{4}{\sqrt[3]{6\,\delta_0^2}}.$$\par
It remains to show that $\tau=-\infty$. 
Firstly, we prove that it is true if 
$h$ is positive or negative definite. Note that in this case $n_in_j=1,$ for  every $i,j =1,2,3$. 
So $\lambda'_i (t)<0,$ for any $i=1,2,3$. Define  $$f(t)=\lambda_1(t)+\lambda_2(t)+\lambda_3(t). $$ Then it is clear that the solution 
exists as long as $f(t)<+\infty$ and, consequently,  $f(\tau)=+\infty.$ Now  observe that by \eqref{flow2} $$
\begin{array}{lcll}
f''(t) &=&-\frac{d}{dt} \left (\sum_{a,b,c} \frac{\lambda_a^2+\lambda_b\lambda_c(t)}{\lambda_1^2\lambda_2^2\lambda_3^2} \right )\\[4pt]
&& =-\sum_{a,b,c}\frac{(2\lambda_a \lambda'_a + \lambda'_b \lambda_c +\lambda_b \lambda_c')\lambda_1^2\lambda_2^2\lambda_3^2-2\sum_{i,j,k}(\lambda_i \lambda'_i (t)\lambda_j^2\lambda_k^2)(\lambda_a^2+\lambda_b\lambda_c)}{\lambda_1^4\lambda_2^4\lambda_3^4},\\
\end{array}
$$
where $(a,b,c)$ and $(i,j,k)\in\{(1,2,3), (2,3,1), (3,1,2)\}.$
But, from \eqref{flow2}, for any  $(a,b,c),$  it follows
$$2\lambda_a \lambda_a' \lambda_1^2\lambda_2^2\lambda_3^2-2(\lambda_a \lambda_a' \lambda_b^2\lambda_c^2)\lambda_a^2=0,$$
and 
\begin{align*}
&\lambda_b' \lambda_c\lambda_1^2\lambda_2^2\lambda_3^2-2(\lambda_b \lambda_b' \lambda_a^2\lambda_c^2)\lambda_b\lambda_c=\lambda_b' \lambda_1^2\lambda_2^2\lambda_3^2(\lambda_c-2\lambda_c)>0,\\
& \lambda_c' \lambda_b\lambda_1^2\lambda_2^2\lambda_3^2-2(\lambda_c \lambda_c' \lambda_a^2\lambda_b^2)\lambda_b\lambda_c= \lambda_c' \lambda_1^2\lambda_2^2\lambda_3^2(\lambda_b-2\lambda_b)>0.
\end{align*}
Therefore  $ f''(t)<0,$ for $t\in (\tau, T)$. 
 But, for $t\in(\tau,0),$ 
$$f (0) -f (t)=\int_t^0 f'(s) ds\geq \int_t^0 f'(0) ds=-t f'(0).$$
Thus, $$
f (t) \leq f (0)+t f'(0),$$
which means $\tau=-\infty$.\par
 In order to prove that $\tau=-\infty$ if $h$ is indefinite, we proceed by contradiction as follows.  Suppose by contradiction that $\tau >-\infty$ and that $\lambda_1(t),\lambda_2(t)$ and $\lambda_3(t)$ are all bounded near $\tau$. Then, we can find a sequence $t_n \to \tau$ for which all $\lambda_i (t_n)$ converge. If the limits of $\lambda_i (t_n)$ are non-zero we can restart the flow past $\tau$, contradicting the maximality of the solution. Therefore, if $\tau>-\infty$, at least one of the $\lambda_i (t_n)$ has to go to zero for $t_n \to \tau$.  Since $\lambda_1 (t) \lambda_2 (t) \lambda_3 (t) = 1/6  *_0(\omega^3_t)={\varepsilon}_t^2$ decreases, we get  also a contradiction. Indeed,
 $$
0  = \lambda_1 (\tau) \lambda_2 (\tau) \lambda_3 (\tau)  = \lim_{t \to \tau} \varepsilon_t^2 \geq \epsilon_0^2>0.
$$
Therefore, if 
$\tau> - \infty$, there is at least one $\lambda_{i}(t)$ $(i=1, 2, 3)$ which is
unbounded. Suppose now that $\lambda_2(t)$ is unbounded. Then, choosing a sequence of negative times $\left\{t_n\right\}$ converging to $\tau$  and such that $\lambda_2(t_{n})-\lambda_2(t_{n-1})$ diverges, it follows that
$$
\begin{array}{lcl}
\lambda_2(t_n)-\lambda_2(t_{n-1})&=&-\int_{t_{n-1}}^{t_n}\frac{\lambda_1 (s)\lambda_3(s)-\lambda_2(s)^2}{\lambda_1(s)^2\lambda_2(s)^2\lambda_3(s)^2}ds\\[4pt]
&=& -\int_{t_{n-1}}^{t_n} \left (\frac{1}{\lambda_1(s)\lambda_3(s)\lambda_2(s)^2}-\frac{1}{\lambda_1(s)^2\lambda_3(s)^2} \right)ds\\[4pt]
&=&\left(-\frac{1}{\varepsilon^2(\bar{t}_n)\lambda_2(\bar{t}_n)}+\frac{1}{\lambda_1^2(\bar{t}_n)\lambda_3^2(\bar{t}_n)}\right)(t_n-t_{n-1})\rightarrow+\infty,
\end{array}
$$
where $\bar{t}_n\in(t_n,t_{n-1})$. 
Hence $\lambda_1(\bar{t}_n)\lambda_3(\bar{t}_n)\rightarrow 0$.
Indeed ${\varepsilon(\bar{t}_n)^2}\lambda_2(\bar{t}_n)$  stays away from zero. But for $n$ large we get a contradiction
$$0>\frac{\lambda_2(t_n)-\lambda_2(t_{n-1})}{t_n-t_{n-1}}=-\frac{\lambda_1(\bar{t}_n)\lambda_3(\bar{t}_n)-\lambda_2^2(\bar{t}_n)}{\lambda_1^2(\bar{t}_n)\lambda_2^2(\bar{t}_n)\lambda_3^2(\bar{t}_n)}>0.$$
 Thus $\lambda_2(t)$ must be bounded. 
The same argument shows that $\lambda_3(t)$ must be bounded as well. So, the only possibility is that $\lambda_1(t)$  is unbounded whereas both $\lambda_2(t)$
and $\lambda_3(t)$ are bounded. As done previously choose $\left\{t_n\right\}$ so that $t_n\rightarrow\tau$ and
$$\lambda_1(t_n)-\lambda_1(t_{n-1})=-\int_{t_{n-1}}^{t_n}\frac{\lambda_2 (s) \lambda_3(s) +\lambda_1(s)^2}{\lambda_1(s)^2\lambda_2(s)^2\lambda_3(s)^2}ds\rightarrow+\infty.$$
Then there exists $\bar{t}_n\in(t_n,t_{n-1})$ such that $\lambda_2(\bar{t}_n)\lambda_3(\bar{t}_n)\rightarrow 0$. We can certainly 
assume that $\lambda_2(\bar{t}_n)\lambda_3(\bar{t}_n)$ decreases in $n$ by choosing a suitable subsequence which we  will  still  denote by $t_n$.  Then,
$$0>\lambda_2 (t_n) \lambda_3({t}_n)- \lambda_2 (t_{n-1} \lambda_3)({t}_{n-1})=\left(\frac{d}{dt}(\lambda_2\lambda_3)\right)  (s_n) (t_n-t_{n-1}),$$
for some $s_n\in(t_n,t_{n-1})$. On the other hand, by \eqref{flow2}, it turns out that 
$$\frac{d}{dt} \left (\lambda_2 (t) \lambda_3(t) \right) =-\frac{\lambda_1 (t)(\lambda_2(t)^2+\lambda_3(t)^2)-\lambda_2(t)\lambda_3(t)(\lambda_2(t)+\lambda_3(t))}{\lambda_1(t)^2\lambda_2(t)^2\lambda_3(t)^2}.$$
Since  
$$\lambda_1(t_n)\rightarrow+\infty,\quad\frac{\lambda_2(t_n)\lambda_3(t_n)(\lambda_2({t}_n)+\lambda_3(t_n))}{\lambda_2^2(t_n)+\lambda_3^2(t_n)}\rightarrow 0,$$
we obtain that  $\frac{d}{dt} \left(\lambda_2\lambda_3\right)(s_n)<0,$ for $n$ large. 
Then we get the following contradiction:
 $$0>\lambda_2\lambda_3({t}_n)-\lambda_2\lambda_3({t}_{n-1})=\left(\frac{d}{dt}\lambda_2\lambda_3\right)(s_n)(t_n-t_{n-1})>0.$$
Thus also $\lambda_1(t)$ must be bounded. But we have already proved that, assuming $\tau>-\infty$, at least one $\lambda$ must be unbounded. To avoid any contradiction it must be $\tau=-\infty$. This completes the proof.
\end{proof}

We  now   solve the coflow \eqref{co-flow} on the $7$-dimensional Heisenberg  group 
when the initial coclosed $\Gtwo$ form is equal to  $\varphi_i$ $(i=1, 2)$,  where $\varphi_1$ and $\varphi_2$ are defined by 
\begin{equation}\label{ex1}
\varphi_1=e^{127}+e^{347}+e^{567}+e^{135}-e^{146} -e^{236}-e^{245}.
\end{equation} and 
\begin{equation} \label{ex3}
\varphi_2=e^{127}-e^{347}-e^{567}+e^{135}-e^{146}+e^{236}+e^{245},
\end{equation}
respectively.  Note  that $\varphi_1$ and $\varphi_2$ induce the same metric and orientation, namely they are $\mathrm{SO}(7)$-equivalent via the special orthogonal transformation
$$
R = {\mbox{diag}}(1,1,1, -1,-1,-1,1).
$$
Moreover,   their dual $4$-forms are given respectively by the closed forms
$$
\star \varphi_1= e^{1234} + e^{1256} + e^{3456} - e^{2467} + e^{1367} + e^{1457} + e^{2357}
$$
and
$$
\star \varphi_2= -e^{1234} - e^{1256} + e^{3456} - e^{2467} - e^{1367} - e^{1457} + e^{2357}.
$$

We will show in the next section that the behavour of the solution for the modified coflow is different.

\begin{corollary} \label{th5.1} The 
 solution of the Laplacian coflow \eqref{co-flow} on $H$ 
with the initial coclosed $\Gtwo$ form $\varphi_1$, defined by \eqref{ex1}, is given by 
\begin{equation}\label{sol1-1}
\varphi(t)=\frac{1}{y(t)}\,(e^{127}+e^{347}+e^{567})+y(t)^3\,(e^{135}-e^{146}-e^{236}-e^{245}), \qquad  t\in \left (-\infty,\,\frac{3}{5}\right),
\end{equation}
where $y=y(t)$ is the positive function
\begin{equation}\label{sol1}
y(t)=\sqrt[10]{1-\frac{5}{3}t}.
\end{equation}
The underlying metrics $g_t$ of this solution converge smoothly, up to pull-back by time-dependent diffeomorphisms, to a flat metric, uniformly on compact sets 
in $H$ as t goes to $- \infty$. 
\end{corollary}

\begin{proof} 
For each $t\in(-\infty, \frac{3}{5})$, we consider the basis $\{f^1(t),\dots,f^7(t)\}$ of left invariant 1-forms
on $H$ defined by
\begin{equation}\label{eqn:new basis}
\begin{aligned}
f^i&=f^i(t)=y(t)\,e^i,\qquad 1\leq i\leq 6, \\
f^7&=f^{7}(t)=y(t)^{-3}\,e^7, 
\end{aligned}
\end{equation}
where the function $y=y(t)$ is given by \eqref{sol1}. Then, $f^i(0)=e^i$, for $i\in\{1, \cdots, 7\}$, and the structure equations of $H$, with respect to 
the basis $\{f^1(t),\dots,f^7(t)\}$, are 
\begin{equation} \label{new-eq:H}
df^i=0, \quad 1\leq i\leq 6,   \qquad df^7=\frac{\sqrt{6}}{6}\,y(t)^{-5}(f^{12}+f^{34}+f^{56}).
\end{equation}
Now, for any $t$, the 3-form $\varphi(t)$ defined by \eqref{sol1-1} has the following expression
\begin{equation}\label{sol-H-1-1}
\begin{aligned}
\varphi(t)&=f^{127}+f^{347}+f^{567}+f^{135}-f^{146} -f^{236}-f^{245}.
\end{aligned}
\end{equation}
Note that $\varphi(0)=\varphi_{1}$ and, for any $t$,  the 3-form $\varphi(t)$ on $H$ induces the metric 
$g_t$ such that the coframe $\{f^1(t),\dots,f^7(t)\}$ of $\frak {h}^*$ is orthonormal. Denote by $\star_{t}$ the Hodge  star operator
determined by $g_t$. Using \eqref{eqn:3-forma G2}, \eqref{eqn:4-forma G2} and \eqref{new-eq:H}, we have $d\,\star_{t}\varphi(t)=0$, where
the 4-form 
$$
\star_{t}\,\varphi(t)=f^{1234}+f^{1256}+f^{1367}+f^{1457}+f^{2357}-f^{2467}+f^{3456}.
$$
So, in terms of the coframe $\{e^1,\dots, e^7\}$ of $\frak {h}^*$, $\star_{t}\,\varphi(t)$ has the following expression
\begin{equation*}
\begin{aligned}
\star_{t}\,\varphi(t)&= y(t)^4(e^{1234}+e^{1256}+e^{3456}) + e^{1367} + e^{1457} + e^{2357} -e^{2467}.
\end{aligned}
\end{equation*}
Thus,
\begin{eqnarray} \label{eq:5-1}
\frac{d}{dt} \left (\star_{t}\,\varphi(t) \right )= 4y(t)^3\,y'(t)\,(e^{1234}+e^{1256}+e^{3456}).
\end{eqnarray}
Moreover, using \eqref{new-eq:H} and \eqref{sol-H-1-1}, we have
$$
\begin{array}{lcl}
-\Delta_t\star_{t}\varphi(t)&=&-d\star_{t}d \varphi(t)=-\frac{\sqrt{6}}{3} y(t)^{-5}\, d\star_{t}\left(f^{1234}+f^{1256}+f^{3456}\right)\\[3pt]
&=&-\frac{2}{3} y(t)^{-10} \left(f^{1234}+f^{1256}+f^{3456}\right),\\[3pt]
\end{array}
$$
or, equivalently,
$$
\begin{array}{lcl}
-\Delta_t\star_{t}\varphi(t)&=& -\frac{2}{3} y(t)^{-6} \left(e^{1234}+e^{1256}+e^{3456}\right).
\end{array}
$$
The last equality and \eqref{eq:5-1} prove that  \eqref{sol1-1} is the solution of the coflow \eqref{co-flow} when the function $y=y(t)$ is given by \eqref{sol1}.

We study the behavior of the underlying metric $g_t$ of the solution $\varphi(t)$ in the limit for  $t  \to - \infty$.  The limit 
can be computed fixing the $\Gtwo$-structure and changing
the Lie bracket as in \cite{La11}.  If we evolve the Lie brackets $\mu (t)$  instead of the 3-form
defining the $\Gtwo$-structure,  the corresponding bracket flow has
a solution for every $t$. Indeed, if we fix on $\R^7$  the 3-form $f^{127}+f^{347}+f^{567}+f^{135}-f^{146} -f^{236}-f^{245},$  then
 the basis  $(f_1(t), \ldots, f_7(t))$ defines, for every  $t <3/5$,
a nilpotent Lie algebra with bracket  $\mu(t)$ such that $\mu(0)$  is the Lie bracket of  $\frak h$.
Moreover, the solution converges to the null bracket corresponding to the abelian
Lie algebra. 
For this, let $\{f_{1}(t),\dots,f_{7}(t)\}$ be the basis dual to $\{f^1(t),\dots,f^7(t)\}$ (defined by \eqref{eqn:new basis}). Then,  
the equations \eqref{new-eq:H} imply that all the Lie brackets $[f_{i}(t), f_{j}(t)]$ $(1\leq i \leq j \leq 7)$ vanish excepting  
$$
[f_{1}(t), f_{2}(t)] = [f_{3}(t), f_{4}(t)] = [f_{5}(t), f_{6}(t)] = - \frac{\sqrt{6}}{6}\,y(t)^{-5} f_{7}(t).
$$
Thus, all the Lie brackets $[f_{i}(t), f_{j}(t)]$ tend to zero as $t$ goes to $- \infty$.
\end{proof}

\medskip
In a similar way   we can prove  the following 

\begin{corollary} \label{th5.2} 
The solution of the Laplacian coflow  \eqref{co-flow} on $H$ 
with initial coclosed $\Gtwo$ form $\varphi_2$, defined by \eqref{ex3},  
is ancient and it is given by
\begin{equation}\label{sol2-1}
 \varphi(t) = \frac{y(t)}{z(t)^2}\,e^{127}-\frac{1}{y(t)}\,e^{347}-\frac{1}{y(t)}\,e^{567}+y(t)z(t)^2\left(e^{135}-e^{146}+e^{236}+e^{245}\right),
\end{equation}
where the functions $y=y(t)$ and $z=z(t)$ satisfy
\begin{equation}\label{sol2}
\begin{cases}
\frac{d}{dt}y(t)=-\frac{1}{12}\frac{y(t)^4+z(t)^4}{y(t)^5 z(t)^8}, \,& \,\, \frac{d}{dt}z(t) = \frac{1}{12}\frac{z(t)^2-y(t)^2}{y(t)^4 z(t)^7},\\
y(0)=1, \,& \,\,   z(0)=1.
\end{cases}
\end{equation}
\end{corollary}

\section{Explicit solutions for the modified Laplacian coflow}

 We study the modified Laplacian coflow \eqref{modifiedcoflow} 
for each of the coclosed $\Gtwo$ forms $\varphi_i$, $i =1,2$, defined respectively by \eqref{ex1} and \eqref{ex3}, on the $7$-dimensional Heisenberg group.
In particular, we prove that the solution of \eqref{modifiedcoflow} for $\varphi_1$ is ancient only
 if the positive constant $A$, that appears in \eqref{modifiedcoflow}, take values in a certain open interval, while 
 the solution of \eqref{modifiedcoflow} for $-\varphi_1$ is ancient for any $A$.
However, we prove that the solution of \eqref{modifiedcoflow} for  $\varphi_2$ is   never ancient.

\begin{theorem}  \label{Th6.1} 
The solution of the modified Laplacian coflow \eqref{modifiedcoflow} for the coclosed $\Gtwo$ form $\varphi_1$, defined by \eqref{ex1}, is given by
\begin{equation}\label{sol_grig_1}
\varphi(t)=\frac{1}{y(t)}\left(e^{127}+e^{347}+e^{567}\right)+y(t)^3\left(e^{135}-e^{146}-e^{236}-e^{245}\right),
\end{equation}
where the function  $y=y(t)$ satisfies 
\begin{equation}\label{ode1}
\begin{cases}\frac{d}{dt}y(t)=\frac{2A\sqrt{6}\,y(t)^5\,-\,1}{12\, y(t)^9},\\ 
y(0)=1.\end{cases}
\end{equation}
Moreover,
\begin{enumerate}
\item [i)] if $0< A<\frac{1}{2\sqrt{6}}$, then $t \in (- \infty, T)$, with $T= - \frac{1}{10\,A^2} \Big(2 \sqrt{6}\, A + \log \big(1 - 2\, \sqrt{6} \,A\big)\Big) >0$. 
Therefore, in this case, the solution \eqref{sol_grig_1} is ancient\rm{;}
\item [ii)] if $A\geq \frac{1}{2\sqrt{6}}$, then $t \in (- \infty, + \infty)$\rm{,} that is, the solution \eqref{sol_grig_1} is eternal\rm{.}
\end{enumerate}
\end{theorem}
\begin{proof} 
By the  Picard-Lindel\"of Theorem, there exists a maximal open interval $I$, containing $0$, and a smooth  function 
$y: I \rightarrow (0, + \infty)$, which is  
the unique  solution of  \eqref{ode1}.

To prove that \eqref{sol_grig_1} is the solution to the coflow \eqref{modifiedcoflow} for $\varphi_1$, we proceed as follows.
As in the proof of Theorem \ref{th5.1} , for each $t\in I$, we consider the basis $\{f^1(t),\dots,f^7(t)\}$ of left invariant 1-forms
on $H$ defined by
\begin{equation*}
\begin{aligned}
f^i&=f^i(t)=y(t)\,e^i,\quad i=1,\dots,6,\\
f^7&=f^7(t)=y(t)^{-3}\,e^7,\\
\end{aligned}
\end{equation*}
where the function $y=y(t)$ now satisfies  \eqref{ode1}. 
Then,  $f^i(0)=e^i$, for $i \in \{1, \cdots, 7\}$, and  the structure equations of $H$, with respect to  the basis $\{f^1(t),\dots,f^7(t)\}$,  are 
\begin{equation} \label{strcon_grig_1}
df^i=0, \quad 1\leq i\leq 6,   \qquad df^7=\frac{\sqrt{6}}{6}\,y^{-5}(t)(f^{12}+f^{34}+f^{56}).
\end{equation}
Moreover, for any $t \in I$, the 3-form $\varphi(t)$ defined by \eqref{sol_grig_1} has the following expression
\begin{equation}\label{sol:th6}
\begin{aligned}
\varphi(t)&=f^{127}+f^{347}+f^{567}+f^{135}-f^{146} -f^{236}-f^{245}.
\end{aligned}
\end{equation}
So, $\varphi(0)=\varphi_{1}$ and, for any $t \in I$,  the 3-form $\varphi(t)$ on $H$ induces the metric 
$g_t$ such that  the coframe $\{f^1(t),\dots,f^7(t)\}$ of $\frak {h}^*$ is orthonormal. Denote by $\star_{t}$ the Hodge  star operator 
determined by $g_t$. Using \eqref{eqn:3-forma G2}, \eqref{eqn:4-forma G2} and \eqref{strcon_grig_1}, we have $d\,\star_{t}\varphi(t)=0$, where
$\star_{t}\varphi(t)$ is given by
$$
\star_{t}\varphi(t)=f^{1234}+f^{1256}+f^{1367}+f^{1457}+f^{2357}-f^{2467}+f^{3456}.
$$
Thus, in terms of the coframe $\{e^1,\dots, e^7\}$ of $\frak {h}^*$, the 4-form $\star_{t}\varphi(t)$ has the following expression
\begin{equation*}
\begin{aligned}
\star_{t}\varphi(t)&=  y(t)^4(e^{1234} + e^{1256}+e^{3456}) + e^{1367} + e^{1457} + e^{2357} - e^{2467}.
\end{aligned}
\end{equation*}
This implies
\begin{eqnarray*}
\frac{d}{dt}\star_t\varphi (t)=4\,y(t)^3\, y'(t)(e^{1234} + e^{1256}+e^{3456}),
\end{eqnarray*}
that is
\begin{eqnarray}\label{thm-6-1}
\frac{d}{dt}\star_t\varphi (t)=\frac{2A\sqrt{6}\,y(t)^5-1}{3\,y(t)^6}(e^{1234} + e^{1256}+e^{3456}),
\end{eqnarray}
since the function $y=y(t)$ satisfies \eqref{ode1}.

On the other hand, by \eqref{torsion} we know that the torsion forms $\tau_i(t)$ $(i=0,1,2,3)$ of $\varphi(t)$
are such that $\tau_1(t)=0=\tau_2(t)$ since $d(\star_t\varphi (t))=0$. Then, from \eqref{strcon_grig_1}, \eqref{sol:th6} and \eqref{torsion}, we have 
\begin{equation}\label{exp:dif-1}
d\varphi (t)=\frac{\sqrt{6}}{3\,y(t)^5}\,(f^{1234}+f^{1256}+f^{3456})=\tau_0(t)\star_t\varphi(t)+\star_t\tau_3(t),
\end{equation}
where 
\begin{gather*}
\begin{array}{lcl}
\tau_3(t)&=&\frac{\sqrt{6}}{7\,y(t)^5}(-f^{135}+f^{146}+f^{236}+f^{245})+\frac{4\sqrt{6}}{21y(t)^5}(f^{127}+f^{347}+ f^{567}),\\
\end{array}
\end{gather*}
\begin{gather*}
\begin{array}{lcl}
\star_t\tau_3(t)&=&\frac{\sqrt{6}}{7y(t)^5}(-f^{1367}-f^{1457}-f^{2357}+f^{2467})+\frac{4\sqrt{6}}{21y(t)^5}(f^{1234}+f^{1256}+f^{3456}),
\end{array}
\end{gather*}
and
$$
\tau_0(t)=\frac{\sqrt{6}}{7y(t)^5}.
$$
 So, according with the first equality of \eqref{exp:dif-1},
\begin{eqnarray*}
\begin{array}{lcl}
 \Delta_t\star_t\varphi (t)+2d\Big((A-\frac{7}{4}\tau_0)\varphi(t)\Big)&=&
  d\star_t d(\varphi(t)) +2(A-\frac{7}{4}\tau_0) d\varphi(t) \\&=&
 \frac{2A\sqrt{6}\,y(t)^5-1}{3\,y(t)^{10}}\left(f^{1234}+f^{1256}+f^{3456}\right),
\end{array}
\end{eqnarray*}
that is
\begin{eqnarray*}
\begin{array}{lcl}
 \Delta_t\star_t\varphi (t)+2d\Big((A-\frac{7}{4}\tau_0)\varphi(t)\Big)&=& \frac{2A\sqrt{6}\,y(t)^5-1} {3\,y(t)^6}\left(e^{1234}+e^{1256}+e^{3456}\right).
\end{array}
\end{eqnarray*}
The last equality, together with \eqref{exp:trace} and \eqref{thm-6-1}, show that \eqref{sol_grig_1} solves  the modified Laplacian 
coflow \eqref{modifiedcoflow} for $\varphi_1$.\par

In order to show that the solution $\varphi(t)$, given by \eqref{sol_grig_1}, is ancient, 
we analyse the behaviour of the function $y=y(t)$ according with the values of the positive constant $A$.
If $A = \frac{1}{2 \sqrt{6}}$, then ${y}(t)\equiv 1$ solves \eqref{ode1} for all $t \in (- \infty, + \infty)$. 
Assume $A \neq  \frac{1}{2 \sqrt{6}}$ and observe that the constant function $\widehat{y}(t)\equiv \left(2\sqrt{6}A\right)^{-1/5}$ satisfies the differential equation that appears in \eqref{ode1}, which 
is autonomous. Consequently any solution ${y}(t)$ having $y'(t_0)=0$ at some time $t_0$ satisfies $y(t_0)=\widehat{y}(t_0)$, giving $y\equiv \widehat{y}$. Hence, the solution $y=y(t)$ of the system \eqref{ode1} is monotone and it must satisfy either $y(t)>\widehat{y}(t)$ 
or $y(t)<\widehat{y}(t)$ for any $t\in I$, according to the value of $A$. In other words,  if $2 \sqrt{6}\, A<1$ then $y(0)<\widehat{y}(0)$, so $y(t)<\widehat{y}(t)$, and similarly  
$y(t)>\widehat{y}(t)$ if $2 \sqrt{6} \,A>1$.

Now, we rewrite the differential equation  that appears in \eqref{ode1} as
$$
\left (\frac{\sqrt{6}}{A}  y(t)^4 + \frac{\sqrt{6}}{A} \frac{y(t)^4}{2 \sqrt{6} A y(t)^5 - 1} \right ) y'(t)= 1.
$$ 
Integrating this equation from $0$ to $t$, we have
\begin{equation}\label{energy1}
t = \frac{\sqrt{6}}{5A} (y(t)^5 - 1) + \frac{1}{10 A^2} \log  \left \vert \frac{1 - 2 \sqrt{6} A y(t)^5}{1 - 2 \sqrt{6} A}  \right \vert.
\end{equation}
This equation allows us to understand the behaviour of the solution at its singular times. Indeed the limits of $y(t)$ must be singular values of \eqref{energy1}; otherwise, through  a trivial compactness argument, we could restart the flow, violating the maximality of solutions.  So, if ${2 \sqrt{6}}\, A < 1$ then $y=y(t)$ decreases from $\left(2\sqrt{6}\,A\right)^{-1/5}$ to $0$ as $t$ goes from 
$-\infty$ to $-\frac{2A\,\sqrt{6}+\mathrm{log}\left(1-2\sqrt{6}\,A\right)}{10A^2}$. Otherwise, if $2 \sqrt{6}\, A>1$, then $y=y(t)$, which now is an increasing function, goes from $\left(2\sqrt{6}\,A\right)^{-1/5}$ to $+\infty$ 
as $t$ goes from $-\infty$ to $+\infty$. In particular, we have that the definition interval $I$ of the function $y = y (t)$  is
$$
\begin{array}{l}
I = (- \infty, - \frac{2 \sqrt{6}\, A + \log(1 - 2 \sqrt{6}\, A)}{10 A^2}), \quad \rm{if} \,\,  A < \frac{1}{2 \sqrt{6}}, \\[3pt]
\end{array}
$$
and
$$
\begin{array}{l}
I = (- \infty, + \infty), \quad \rm{if} \,\,  A \geq \frac{1}{2 \sqrt{6}}.
\end{array}
$$
\end{proof}

\begin{remark} 
In a similar way as in the proof  of Theorem \ref{th5.1}, one can check that the Riemannian curvature $R(g_t)$
of the metric $g_t$ induced by \eqref{sol_grig_1} is such that
$$
\| R(g_t) \|_{g_t} ^2 = \frac{23}{48} y(t)^{-20},
 $$
and  so, in the  case iii) (corresponding to  $A>\frac{1}{2\sqrt{6}}$) 
$\lim_{ t \to + \infty}  R(g_t) = 0.$
\end{remark} 

\medskip 

In the following theorem we study the  modified Laplacian coflow  \eqref{modifiedcoflow} when the initial coclosed $\Gtwo$ form 
on the 7-dimensional Heisenberg group is equal to
$-\varphi_1$, where $\varphi_1$ is defined by \eqref{ex1}.

\begin{theorem}  \label{Th6.1-2} 
The solution of the modified Laplacian coflow \eqref{modifiedcoflow} with initial coclosed $\Gtwo$ form $-\varphi_1$ is ancient and it is given by
\begin{equation}\label{sol_grig_2}
\varphi(t)=-\frac{1}{y(t)}\left(e^{127}+e^{347}+e^{567}\right)-y(t)^3\left(e^{135}-e^{146}-e^{236}-e^{245}\right),
\end{equation}
where $t \in (- \infty, T)$, with $T=\frac{\sqrt{6}}{5A}\Big(1- (2A\sqrt{6})^{-1}\,{\mathrm {log}}\big(2A\sqrt{6}+1\big)\Big)$, and
the  function  $y=y(t)$  satisfies 
\begin{equation}\label{ode2}
\begin{cases}\frac{d}{dt}y(t)=-\frac{2A\sqrt{6}y(t)^5+1}{12 y(t)^9},\\ y(0)=1.\end{cases}
\end{equation}
The underlying metrics $g_t$ of this solution converge smoothly, up to pull-back by time-dependent diffeomorphisms, to a flat metric, 
uniformly on compact sets in $H$ as t goes to $- \infty$.
\end{theorem}

\begin{proof}
By the  Picard-Lindel\"of Theorem, there exists  a maximal open interval $I$, containing $0$, and a smooth function 
$y: I \rightarrow (0, + \infty)$, which is the unique solution of  \eqref{ode2}.\par
To prove that  \eqref{sol_grig_2} is the solution of the coflow \eqref{modifiedcoflow} for $-\varphi_1$, we proceed as follows.
As in the proof of Theorem \ref{Th6.1}, for each $t\in I$, we consider the basis $\{f^1(t),\dots,f^7(t)\}$ of left invariant 1-forms
on $H$ defined by
\begin{equation*}
\begin{aligned}
f^i&=f^i(t)=y(t)\,e^i,\quad i=1,\dots,6\\
f^7&=f^7(t)=y(t)^{-3}\,e^7,\\
\end{aligned}
\end{equation*}
where the function $y=y(t)$ now satisfies  \eqref{ode2}. 
Then, $f^i(0)=e^i$, for $i\in\{1, \cdots, 7\}$, and the structure equations of $H$, with respect to  the basis $\{f^1(t),\dots,f^7(t)\}$, are 
\begin{equation} \label{strcon_grig_2}
df^i=0, \quad 1\leq i\leq 6,   \qquad df^7=\frac{\sqrt{6}}{6}\,y(t)^{-5}(f^{12}+f^{34}+f^{56}).
\end{equation}
Now, for any $t\in I$, the 3-form $\varphi(t)$ defined by \eqref{sol_grig_2} has the following expression
\begin{equation}\label{exp:thm6.1-2}
\begin{aligned}
\varphi(t)&=-(f^{127}+f^{347}+f^{567}+f^{135}-f^{146} -f^{236}-f^{245}).
\end{aligned}
\end{equation}
So, $\varphi(0)=-\varphi_{1}$ and, for any $t\in I$,  the metric $g_t$ induced by $\varphi(t)$ 
is such that the  coframe $\{f^1(t),\dots,f^7(t)\}$ of $\frak {h}^*$ is orthonormal. Denote by $\star_{t}$ the Hodge star operator 
determined by $g_t$. Using  \eqref{strcon_grig_2}, we have $d\star_{t}\varphi(t)=0$, where
$\star_{t}\,\varphi(t)$ is given by
$$
\star_{t}\,\varphi(t)=f^{1234}+f^{1256}+f^{1367}+f^{1457}+f^{2357}-f^{2467}+f^{3456}.
$$
Then, in terms of the coframe $\{e^1,\dots, e^7\}$ of $\frak {h}^*,$ the 4-form $\star_{t}\,\varphi(t)$ has the following expression
\begin{equation*}
\begin{aligned}
\star_{t}\,\varphi(t)&= y(t)^4(e^{1234} + e^{1256}+e^{3456}) + e^{1367} + e^{1457}+e^{2357}-e^{2467}.
\end{aligned}
\end{equation*}
Therefore,
\begin{eqnarray*}\label{ddt2}
\begin{array}{lcl}
\frac{d}{dt}\star_t\varphi(t)=4y(t)^3\,y'(t)(e^{1234} + e^{1256}+e^{3456}),\\
\end{array}
\end{eqnarray*}
that is, 
\begin{eqnarray}\label{ddt2}
\begin{array}{lcl}
\frac{d}{dt}\star_t\varphi(t)=-\frac{2A\sqrt{6}\,y(t)^5+1}{3\,y(t)^6}(e^{1234} + e^{1256}+e^{3456}),
\end{array}
\end{eqnarray}
since the function $y=y(t)$ satisfies \eqref{ode2}.

On the other hand, by \eqref{torsion} we know that the torsion forms $\tau_i(t)$ $(i=0,1,2,3)$ of $\varphi(t)$
are such that $\tau_1(t)=0=\tau_2(t)$ since $d(\star_t\varphi (t))=0$. Then, from \eqref{strcon_grig_2}, \eqref{exp:thm6.1-2} and using again \eqref{torsion}, we have 
\begin{equation}\label{exp:dif-2}
d\varphi (t)=-\frac{\sqrt{6}}{3\,y(t)^5}\,(f^{1234}+f^{1256}+f^{3456})=\tau_0(t)\star_t\varphi(t)+\star_t\tau_3(t),
\end{equation}
where
\begin{gather*}
\begin{array}{lcl}
\tau_3(t)&=&\frac{\sqrt{6}}{7\,y(t)^5}(-f^{135}+f^{146}+f^{236}+f^{245})+\frac{4\sqrt{6}}{21y(t)^5}(f^{127}+f^{347}+ f^{567}),\\
\end{array}
\end{gather*}
\begin{gather*}
\begin{array}{lcl}
\star_t\tau_3(t)&=&\frac{\sqrt{6}}{7y(t)^5}(-f^{1367}-f^{1457}-f^{2357}+f^{2467})+\frac{4\sqrt{6}}{21y(t)^5}(f^{1234}+f^{1256}+f^{3456}),
\end{array}
\end{gather*}
and
$$
\tau_0(t)=-\frac{\sqrt{6}}{7y(t)^5}.
$$
Then, according with the first equality of \eqref{exp:dif-2},
\begin{eqnarray*}
\begin{array}{lcl}
\Delta_t\star_t\varphi (t)+2d\Big((A-\frac{7}{4}\tau_0)\varphi(t)\Big)&=&d\star_t d(\varphi(t)) +2(A-\frac{7}{4}\tau_0) d\varphi(t) \\&=&
-\frac{2A\sqrt{6}\,y(t)^5+1}{3y(t)^{10}}\left(f^{1234}+f^{1256}+f^{3456}\right),
\end{array}
\end{eqnarray*}
or, equivalently,
\begin{eqnarray*}
\begin{array}{lcl}
 \Delta_t\star_t\varphi (t)+2d\Big((A-\frac{7}{4}\tau_0)\varphi(t)\Big)&=&-\frac{2A\sqrt{6}\,y(t)^5+1}{3y(t)^{6}}\left(e^{e^{1234}+1256}+e^{3456}\right).
\end{array}
\end{eqnarray*}
The last equality, together with  \eqref{exp:trace} and \eqref{ddt2}, show that \eqref{sol_grig_2} solves the modified Laplacian flow \eqref{modifiedcoflow} 
for $-\varphi_1$.\par

To show that the solution $\varphi(t)$, given by \eqref{sol_grig_2}, is ancient, we study 
 the behaviour of the function $y=y(t)$. To this end, we rewrite the differential equation that appears in  \eqref{ode2}  as
$$-\frac{12\,y(t)^9}{2A\sqrt{6}\,y(t)^5+1}\,y'=1.$$
Integrating this equation  from $0$ to $t$ we obtain
\begin{equation}\label{energy2}
\frac{\sqrt{6}}{5A}\left(1-y^5(t)\right)+\frac{1}{10A^2}\mathrm{log}\left(\frac{2A\sqrt{6}\,y^5(t)+1}{2A\sqrt{6}+1}\right)=t.
\end{equation}
Clearly $y'(t) < 0$ since the function $y=y(t)$ satisfies  the differential equation that appears in \eqref{ode2}. Then, \eqref{energy2} implies  
that  the function $y=y(t)$ decreases from $+\infty$  to $0$ as $t$ goes from $-\infty$ to $\frac{\sqrt{6}}{5A}\Big(1- \frac{1}{2A\sqrt{6}}\,{\mathrm {log}}\big(2A\sqrt{6}+1\big)\Big)$.\par
To study the behaviour of the underlying metric  $g_t$  of
the solution \eqref{sol_grig_2}  for $ t   \to - \infty$,  we proceed in a similar way as in  the proof of Theorem \ref{th5.1}.
\end{proof}

\medskip
Concerning the modified Laplacian coflow \eqref{modifiedcoflow} for the coclosed $\Gtwo$ form $\varphi_2$
on the 7-dimensional Heisenberg group $H$ we have the following.

\begin{theorem}  \label{Th6.2} 
The solution of  the modified Laplacian coflow \eqref{modifiedcoflow} with initial coclosed
$\Gtwo$-structure $\varphi_2$ is defined on a bounded interval, and it is given by
\begin{equation}\label{sol_grig_3}
\varphi(t)=\frac{y(t)}{z(t)^2}e^{127}-{y(t)^{-1}}\left(e^{347}+e^{567}\right)+y(t)z(t)^2\left(e^{135}-e^{146}+e^{236}+e^{245}\right),
\end{equation}
where the functions  $y=y(t)$ and $z=z(t)$ satisfy
\begin{equation}\label{ode3}
\begin{cases}
\frac{d}{dt}y(t)=\frac{2A\sqrt{6}\,y(t)z(t)^6+2z(t)^2+y(t)^2}{12\,y(t)^3z(t)^8},&
\frac{d}{dt}z(t)=-\frac{2A\sqrt{6}\,y(t)z(t)^4+1}{12\,y(t)^2z(t)^7},\\
y(0)=1,&z(0)=1.
\end{cases}
\end{equation}
\end{theorem}

\begin{proof}  
By the Picard–Lindel\"of Theorem, there exists a maximal open interval $I$, containing 0, and two smooth functions $y, z: I \rightarrow (0, + \infty)$, which are the unique solution of \eqref{ode3}.\par
We first prove that \eqref{sol_grig_3} is the solution of the  coflow  \eqref{modifiedcoflow} for $\varphi_2$.  As in the proof of Theorem \ref{th5.2}, for each $t\in I$, we consider the basis $\{f^1(t),\dots,f^7(t)\}$ of left invariant 1-forms
on $H$ defined by
\begin{equation*}
\begin{aligned}
f^i&=f^i(t)=y(t)\,e^i,\qquad\, i=1,2,\\
f^i&=f^i(t)=z(t)\,e^i,\qquad\, i=3,\dots,6,\\
f^7&=f^7(t)=y(t)^{-1}z(t)^{-2}\,e^7,
\end{aligned}
\end{equation*}
where the functions $y=y(t)$ and $z=z(t)$ satisfy now \eqref{ode3}. Then, $f^i(0)=e^i$, for $i\in\{1,\cdots,7\},$  and the structure equations of $H$, with respect to  the basis $\{f^1(t),\dots,f^7(t)\}$, are 
\begin{equation} \label{strcon_grig_3}
\begin{aligned}
df^i&=0, \quad 1\leq i\leq 6, \\ 
df^7&=\frac{\sqrt{6}}{6}\,y(t)^{-1}z(t)^{-2}\Big(y(t)^{-2}f^{12}+z(t)^{-2}f^{34}+z(t)^{-2}f^{56}\Big).
\end{aligned}
\end{equation}
Moreover, for any $t\in I$, the 3-form  $\varphi(t)$ defined by \eqref{sol_grig_3} has the following expression
\begin{equation}\label{sol-H-2-2}
\begin{aligned}
\varphi(t)&=f^{127}-f^{347}-f^{567}+f^{135}-f^{146} +f^{236}+f^{245}.
\end{aligned}
\end{equation}
So $\varphi(0)=\varphi_{2}$ and, for any $t\in I$, the 3-form $\varphi(t)$ on $H$ induces the metric 
$g_t$ such that $\{f^1(t),\dots,f^7(t)\}$ of $\frak {h}^*$ is an orthonormal basis of $\frak {h}^*$. Denote by $\star_{t}$ the Hodge operator 
determined by $g_t$. Using \eqref{eqn:3-forma G2}, \eqref{eqn:4-forma G2} and \eqref{strcon_grig_3}, we have $d\,\star_{t}\varphi(t)=0$, where
 $\star_{t}\varphi(t)$ is given by 
$$
\star_{t}\varphi(t)= -f^{1234}-f^{1256}-f^{1367}-f^{1457}+f^{2357}-f^{2467}+f^{3456}.
$$
Thus, in terms of the coframe $\{e^1,\dots, e^7\}$ of $\frak {h}^*$, the 4-form $\star_{t}\varphi(t)$ has the following expression
\begin{equation*}
\begin{aligned}
\star_{t}\varphi(t)&=  y(t)^2z(t)^2(-e^{1234} - e^{1256}) - e^{1367} - e^{1457} + e^{2357} - e^{2467}+z(t)^4 e^{3456}.
\end{aligned}
\end{equation*}
Therefore,
$$
\begin{array}{lcl}
\frac{d}{dt} \left (\star_{t}\varphi(t) \right) &= & 2\Big(y(t) z(t)^2 y'(t) +y(t)^2 z(t) z'(t)\Big)(-e^{1234} - e^{1256})+4z(t)^3 z'(t)\,e^{3456},\\[3pt]
\end{array}
$$
that is
\begin{equation}\label{thm-6-3-star}
\begin{aligned}
\frac{d}{dt}\star_t\varphi (t)&=\frac{A\sqrt{6}\big(y(t)^3 z(t)^2-y(t) z(t)^4\big)-1}{3\,y(t)^2 z(t)^4}(e^{1234} + e^{1256})\\ &
-\frac{2A\sqrt{6}\,y(t) z(t)^4+1}{3\,y(t)^2 z(t)^4}e^{3456},
\end{aligned}
\end{equation}
since the functions $y=y(t)$ and $z=z(t)$ satisfy \eqref{ode3}.

On the other hand, let us consider the torsion forms $\tau_i(t)$ $(i=0,1,2,3)$ of $\varphi(t)$. By \eqref{torsion},
$\tau_1(t)=0=\tau_2(t)$ since $d(\star_t\varphi (t))=0$. Then, from \eqref{strcon_grig_3}, \eqref{sol-H-2-2} and using again \eqref{torsion}, we have 
\begin{equation}\label{exp:dif-1-0}
\begin{aligned}
d\varphi (t)&=\frac{\sqrt{6}}{6} y(t)^{-1}z(t)^{-2}\Big(\big(z(t)^{-2}-y(t)^{-2}\big)(f^{1234}+f^{1256}) - 2 z(t)^{-2}f^{3456}\Big)\\
&=\tau_0(t)\star_t\varphi(t)+\star_t\tau_3(t),
\end{aligned}
\end{equation}
where 
\begin{align*}
\begin{array}{lcl}
\tau_3 (t) &=&-\frac{\sqrt{6}\,\big(5y(t)^2+z(t)^2\big)}{21\,y(t)^3 z(t)^4}f^{127}+\frac{\sqrt{6}\,\big(3y(t)^2-5z(t)^2\big)}{42\,y(t)^3z(t)^4}(f^{347}+f^{567})\\&&
+\frac{\sqrt{6}\,\big(2y(t)^2-z(t)^2\big)}{21\,y(t)^3 z(t)^4}(f^{135}-f^{146}+f^{236}+f^{245}),
\end{array}
\end{align*}

\begin{align*}
\begin{array}{lcl}
\star_t\tau_3(t)&=&-\frac{\sqrt{6}\,\big(5y(t)^2+z(t)^2\big)}{21\,y(t)^3 z(t)^4}f^{3456}+\frac{\sqrt{6}\,\big(3y(t)^2-5z(t)^2\big)}{42\, y(t)^3 z(t)^4}(f^{1234}+f^{1256})\\
&&+\frac{\sqrt{6}\,\big(2y(t)^2-z(t)^2\big)}{21\, y(t)^3 z(t)^4}(-f^{1367}-f^{1457}+f^{2357}-f^{2467}),
\end{array}
\end{align*}
and
\begin{eqnarray*}
\tau_0 (t)=- \frac{\sqrt{6}}{21y(t)^3 z(t)^4} \big(2y(t)^2-z(t)^2\big).
\end{eqnarray*}
Then,  according with the first equality of \eqref{exp:dif-1-0},
\begin{eqnarray*}
\begin{array}{lcl}
\Delta_t\star_t\varphi (t)+2d\Big((A-\frac{7}{4}\tau_0)\varphi(t)\Big)&=& d\star_t d(\varphi(t)) +2(A-\frac{7}{4}\tau_0) d\varphi(t) \\&=&
\frac{A\sqrt{6}\,\big(y(t)^3 z(t)^2-y(t) z(t)^4\big)-1}{3\,y(t)^4 z(t)^6}\big(f^{1234}+f^{1256}\big)\\
&& -\frac{2A\,\sqrt{6} y(t) z(t)^4+1}{3\,y(t)^2 z(t)^8} f^{3456},
\end{array}
\end{eqnarray*}
or, equivalently,
\begin{eqnarray*}
\begin{array}{lcl}
\Delta_t\star_t\varphi (t)+2d\Big((A-\frac{7}{4}\tau_0)\varphi(t)\Big)&=&\frac{A\sqrt{6}\,\big(y(t)^3 z(t)^2-y(t) z(t)^4\big)-1}{3\,y(t)^2 z(t)^4}\big(e^{1234}+e^{1256}\big)\\
&& -\frac{2A\,\sqrt{6} y(t) z(t)^4+1}{3\,y(t)^2 z(t)^4} e^{3456}.
\end{array}
\end{eqnarray*}
The last equality, together with  \eqref{exp:trace} and \eqref{thm-6-3-star}, show that \eqref{sol_grig_3} solves the modified Laplacian flow \eqref{modifiedcoflow} 
for $\varphi_2$.

To prove that \eqref{sol_grig_3} is defined on a bounded interval, we will show that
$t_+ =\mathrm{sup}(I) < +\infty$ and $t_-=\mathrm{inf}(I) > - \infty$. On the one hand, we know that the functions 
$y=y(t)$ and $z=z(t)$ are positive. Then, the system \eqref{ode3} implies that $z'(t)<0<y'(t)$, for any $t\in I$.  Therefore, 
the function $z=z(t)$ is decreasing, and $y=y(t)$ is increasing. Thus,  there exist
$${\lim_{t\rightarrow t_-}}y(t)=y_-\in [0,1) \,\, \quad\text{and}\,\,\quad{\lim_{t\rightarrow t_+}}z(t)=z_+\in [0,1).$$
Now, using  \eqref{ode3}, it is straightforward to verify that the function $z'' = z''(t)$ satisfies
\begin{equation*}
\begin{array}{lcl}
z''&=&-\frac{1}{144\,y^6z^{15}}\Big(
24A^2(3y^{4}z^{8}-y^{2}z^{10})+2A\sqrt{6}\,(9y^3z^4-4yz^6)+5y^2-4z^2\Big),\\
\end{array}
\end{equation*}
for any $t\in I$. Note that in the last equality, the functions 
$(3y^{4}z^{8}-y^{2}z^{10})=y^2 z^8(3y^2 - z^2)$, $(9y^3z^4-4yz^6) = y z^4(9 y^2 - 4 z^2)$ and $(5y^2-4z^2)$ are positive functions in $(0,t_+)$.
Indeed, their values at $t=0$ are positive, and $z=z(t)$ decreases while $y=y(t)$ increases in $(0,t_+)$.
Therefore, $z''(t)<0$, for $t\in (0,t_+)$. Thus, $z'(t)<z'(0)<0$, for any $t\in (0,t_+)$. 
 Now, we choose a sequence $\left\{t_n\right\}\subset I$ of positive times converging to $t_+$. Then,  
$$z(t_n)-1=\int_0^{t_n}z'(t)\,dt<\int_0^{t_n}z'(0)\,dt<z'(0)\,t_n.$$
So, $t_n<\frac{z(t_n)-1}{z'(0)}$ and, consequently,  $t_+\leq \frac{z_+-1}{z'(0)}<+\infty$.\par
Using again \eqref{ode3}, we have
\begin{equation}\label{yder}
\begin{aligned}
-144 y^7 z^{16} y'' &= 48A^2(z^{12} y^2 - z^{10} y^4) +2A\sqrt{6}\,(10 z^8 y - 11 z^6 y^3 - 8 z^4 y^5)\\&
+ 12 z^4 - 4 z^2 y^2 -7y^4.\\
\end{aligned}
\end{equation}
Then,  it is possible to show that   $y''(t)<0$ in some neighbourhood of $t_-$.   Indeed, the functions 
 $z^{12} y^2 - z^{10} y^4$ and  $$(12 z^4 - 4 z^2 y^2 - 7 y^4) =  4 z^2 (z^2- y^2)+(8z^4-7y^4)$$ are both positive on $(t_-, 0)$,   
 since   the functions $z^2 - y^2$ and  $8z^4-7y^4$ are both decreasing.
 Moreover,  the  solution  is maximal for $t$ going to $t_-$. Therefore,  the limits $\lim_{t\rightarrow t_-} z(t)=z_-$ and  $\lim_{t\rightarrow t_-} y(t)=y_-$ cannot 
 be both finite and   different from zero, otherwise we can  restart the flow. As a consequence, since $y'(t)>0$ and $z'(t) <0$, for any $t \in I$, 
 we  get  that either $z_- < + \infty$ (and  consequently $y_{-}=0$)  or $z_-=+\infty$.\par
In the first case, the leading term  (as polynomial in $z$) of the right side of \eqref{yder} is $12 z^4$, so it must be positive  
in a  neighbourhood of $t_-$.  On the other hand $ -144 y^7z^{16} <0$, so   $y''(t)<0$ in some neighbourhood of $t_-$. 
In  the other case (i.e. when $z_-=+\infty$),   
$$ \lim_{t \to t_-} (10 z^8  - 11 z^6 y^2- 8 z^4 y^4 )  = + \infty$$ since $z_- = + \infty$ and $y$ is bounded. 
Therefore $y(10 z^8  - 11 z^6 y^2- 8 z^4 y^4 )$ is positive  in some neighbourhood of $t_-$. 
Hence, in both cases,  it follows  that $y''<0 $ for $t\in (\overline{t}, t_{-})$, for some $\overline{t}\in(t_-,0)$, i.e. that  
$y' (t) > y'(\overline{t})$, for  $t\in (t_{-}, \overline{t})$. Now, we choose a sequence of negative times
$\left\{t_n\right\}\subset (t_-,\overline{t})$ converging to $t_-$. Then,
$$
y(\overline{t})-y(t_n)=\int_{t_n}^{\overline{t}} y'(t)\,dt > \int_{t_n}^{\overline{t}}  y'(\overline{t})\,dt =(\overline{t}-t_n)\, y'(\overline{t}).
$$
It follows that $t_{n} > \frac{y(t_n)-y(\overline{t})}{y'(\overline{t})}+\overline{t}$. So,
 $t_{-} \geq \frac{y_{-} - y(\overline{t})}{y'(\overline{t})}+\overline{t} > -\infty$.\par
\end{proof}

\section*{Acknowledgements}

We would like to thank Ernesto Buzano for  very helpful conversations and suggestions. 
The first and third authors are supported by the project FIRB ``Geometria differenziale e teoria geometrica delle funzioni'',
 the project PRIN
\lq\lq  Variet\`a reali e complesse: geometria, topologia e analisi armonica" and by G.N.S.A.G.A. of I.N.d.A.M.
The second author is supported through Project MINECO (Spain) 
PGC2018-098409-B-100 and
Basque Government Project IT1094-16.


\begin{thebibliography}{69} 

\bibitem{BFF}  
\textsc{L.~Bagaglini, M.~Fern\'andez, A. Fino}, Coclosed $\Gtwo$-structures inducing nilsolitons, 
 {\em Forum Math.\/} {\bf 30} (2018), 109--128.
 
\bibitem{Bryant} 
\textsc{R. L.~Bryant}, Some remarks on $\Gtwo$-structures, \emph{Proceedings of G\"{o}kova Geometry-Topology Conference 2005}, 
75--109, G\"{o}kova Geometry/Topology Conference (GGT), G\"{o}kova, 2006.

\bibitem{BF} 
\textsc{R. L.~Bryant, F.~Xu}, Laplacian Flow for Closed $\Gtwo$-Structures: Short Time Behavior,
arXiv: 1101.2004 [math.DG].

\bibitem{FFM} 
\textsc{M.~Fern\'andez, A.~Fino, V.~Manero},  Laplacian flow of closed $\Gtwo$-structures inducing
nilsolitons,  {\em J. Geom. Anal.\/} {\bf 26} (2016), 1808--1837.

\bibitem{FG} 
\textsc{M.~Fern\'andez, A.~Gray}, Riemannian manifolds with structure group $\Gtwo$, 
{\em Ann. Mat. Pura Appl.\/}  {\bf 132} (1982), 19--45.

\bibitem{Grigorian} 
\textsc{S.~Grigorian},  Short-time behavior of a modified Laplacian coflow of $\Gtwo$-structures,   
{\em Adv. Math.\/}  {\bf 248} (2013), 378--415.


\bibitem{Hamilton}
\textsc{R.~S.~Hamilton}, Three-manifolds with positive Ricci curvature, 
{\em J. Diff. Geom.\/} {\bf 17} (1982), 255--306.


\bibitem{HL} 
\textsc{R.~Harvey, H. B.~Lawson}, Calibrated geometries, 
{\em Acta Math.\/} {\bf 148} (1982), 47--157.


\bibitem{Hitchin} 
\textsc{N.~Hitchin}, The geometry of three-forms in six and seven dimensions, 
{\em J.  Diff. Geom.\/} {\bf 55} (2000), 547-576.


\bibitem{KMT}  
\textsc{S.~Karigiannis, B.~McKay, M.P.~Tsui}, Soliton solutions for the Laplacian coflow of some 
$\Gtwo$-structures with symmetry, 
{\em Diff. Geom. Appl.\/}  {\bf  30} (2012), 318--333.


\bibitem{La09} 
\textsc{J.~Lauret},  Einstein solvmanifolds and nilsolitons, in New Developments in Lie Theory and
Geometry, Contemp. Math. {\bf 491}, Amer. Math. Soc, Providence, RI, 2009, pp. 1--35.


\bibitem{La11}
\textsc{J.~Lauret}, The Ricci flow for simply connected nilmanifolds, {\em Comm. Anal. Geom.}  {\bf 19} (2011),
no. 5, 831--854.

\bibitem{Lauret3} J. Lauret, Laplacian flow of homogeneous $G_2$-structures and its solitons, {\em Proc. London Math.}
{\bf 114} (2017), 527--560.

\bibitem{LW} 
\textsc{J.~Lotay, Y.~Wei},  
Laplacian flow for closed $\Gtwo$-structures:
Shi-type estimates, uniqueness and compactness,   {\em Geom. Funct. Anal.}  {\bf 27}  (2017), no. 1, 165--233.


\bibitem{LW2} 
\textsc{J.~Lotay, Y.~Wei},  Stability of torsion-free $\Gtwo$-structures along the Laplacian flow,
{\em J.  Differential Geom.} {\bf 111} (2019), no. 3, 495--526.


\bibitem{Lotay} 
 \textsc{J.~Lotay, Y.~Wei}, Laplacian flow for closed $\Gtwo$-structures: real analyticity,  
  {\em Comm. Anal. Geom.} {\bf 27} (2019), 73--109.



\bibitem{Schulte}
\textsc{F. Schulte-Hengesbach},  Half-flat structures on Lie groups, PhD Thesis (2010),
Hamburg, available at http://www.math.uni-hamburg.de/home/schulte- hengesbach/
diss.pdf.

\end{thebibliography}
\end{document}